\documentclass{alggeom}

\usepackage{mathpazo}
\usepackage{euler}

\renewcommand{\contentsname}{Contents}

\usepackage{amssymb,amsmath}
\usepackage{enumitem}
\usepackage{amsthm}
\usepackage{euscript}
\usepackage[usenames,dvipsnames]{color}
\usepackage[T1]{fontenc}
\usepackage[utf8]{inputenc}
\usepackage[singlespacing]{setspace}
\usepackage[colorinlistoftodos]{todonotes}
\usetikzlibrary{positioning}
\usepackage{bookmark}
\usepackage{multirow,bigdelim}
\usepackage{stackengine}
\usepackage{mathrsfs}
\usepackage{mathtools}
\usepackage{tikz-cd}
\usepackage{pgf,tikz,pgfplots}
\pgfplotsset{compat=1.15}
\usepackage{mathrsfs}
\usetikzlibrary{arrows}

\newtheorem{theorem}{Theorem}[section]
\newtheorem{proposition}[theorem]{Proposition}
\newtheorem{lemma}[theorem]{Lemma}
\newtheorem{corollary}[theorem]{Corollary}
\theoremstyle{definition}
\newtheorem{definition}[theorem]{Definition}
\newtheorem{question}[theorem]{Question}
\theoremstyle{remark}
\newtheorem{example}[theorem]{Example}
\theoremstyle{remark}
\newtheorem{remark}[theorem]{Remark}

\DeclareMathOperator{\spec}{Spec}

\DeclareMathOperator{\aut}{Aut}

\DeclareMathOperator{\ord}{ord}

\DeclareMathOperator{\an}{an}

\DeclareMathOperator{\sw}{sw}

\DeclareMathOperator{\Frac}{Frac}

\DeclareMathOperator{\Def}{Def}

\DeclareMathOperator{\Alg}{Alg}

\DeclareMathOperator{\dsw}{dsw}

\DeclareMathOperator{\st}{st}

\DeclareMathOperator{\Proj}{Proj}

\DeclareMathOperator{\rsw}{rsw}

\begin{document}

\title{Hurwitz trees and deformations of Artin-Schreier covers}

\author{Huy Dang}
\email{hqd4bz@virginia.edu}
\address{Institute of Mathematics, Vietnam Academy of Science and Technology, Hanoi, Vietnam}

%

\thanks{The author is supported by NSF-DMS grants 1602054 and 1900396.}

\begin{abstract}
Let $R$ be a complete discrete valuation ring of equal characteristic $p>0$. In this paper, we first generalize the notion of Hurwitz trees in mixed characteristic to Artin-Schreier covers ($\mathbb{Z}/p$-covers in characteristic $p$) of a formal disc over $R$. It is a combinatorial-differential object with the shape of the dual graph of a cover's semi-stable model, and is endowed with essential degeneration data, measured by Kato's refined Swan conductors, of that cover. We then show how the existence of a deformation between two given Artin-Schreier covers is completely determined by the existence of certain Hurwitz trees. Finally, using the Hurwitz trees criteria, we improve a known result about the connectedness of the moduli space of Artin-Schreier curves of fixed genus.
\end{abstract}

\maketitle

\vspace*{6pt}\tableofcontents  

\section{Introduction}

Throughout this paper, we assume that $k$ is an algebraically closed field of characteristic $p>0$. We use the notation $\{ \ldots \}$ to denote a multi-set. An \textit{Artin-Schreier curve} is a smooth, projective, connected $k$-curve $Y$, which is a $\mathbb{Z}/p$-cover of the projective line $\mathbb{P}^1_k$. The moduli space of Artin-Schreier $k$-curves of fixed genus $g$, which we denote by $\mathcal{AS}_g$, has the  property that there are curves with different \textit{branching data} (e.g., different number of branch points) which lie in the same connected component  \cite{MR2985514}. That motivates the study of equal characteristic deformations of these covers. In \cite{DANG2020398}, by explicitly constructing some local deformations, it was shown that $\mathcal{AS}_g$ is connected when $g$ is sufficiently large.

\begin{theorem}[cf. {\cite[Theorem 1.1]{DANG2020398}}]\phantomsection
\label{thmconnected} 
\begin{itemize}[noitemsep]
  \item When $p=3$, $\mathcal{AS}_g$ is always connected.
  \item When $p=5$,  $\mathcal{AS}_g$  is connected for any $g \ge 14$ and $g=0,2$. It is disconnected if $g=4,6,8$.
  \item When $p>5$, $\mathcal{AS}_g$ is connected if $g \ge \frac{(p^3-2p^2+p-8)(p-1)}{8}$ and $g \le \frac{p-1}{2}$. It is disconnected if $\frac{p-1}{2} < g \le \frac{(p-1)^2}{2}$.
\end{itemize}
\end{theorem}
\noindent Observe that, when $p=5$, the only cases that the theorem does not cover are $g=10$ and $g=12$ ($\mathcal{AS}_g$ is empty otherwise). The techniques from \cite{DANG2020398} are inadequate to study these moduli spaces, as well as other cases. In this paper, with the aim to improve that result, we study local deformations in more detail using the notion of \textit{Hurwitz tree}. A quick overview is given below.

Let $R$ be a complete discrete valuation ring with characteristic $p>0$ residue field. When $R$ is of mixed characteristic (for example $R=W(k)$ where $k$ is a perfect field of characteristic $p>0$), a $G$-cover of the formal disc $\spec R[[X]]$ gives rise to a combinatorial object called a \textit{Hurwitz tree}, which has the shape of the dual graph of the semistable model associated to the cover (see e.g., \cite{MR2534115}, \cite{2000math.....11098H}). The existence of such a tree, along with some other conditions, is necessary for a cover in characteristic $p$ to "lift" to characteristic $0$. In particular, when $G\cong \mathbb{Z}/p$, Henrio proves that the lifts of a $G$-cover can be classified by Hurwitz trees of certain forms associated to the cover \cite{2000math.....11098H}. Moreover, by generalizing Henrio's technique, Bouw and Wewers prove an analog of Henrio's lifting result for $G \cong \mathbb{Z}/p \rtimes \mathbb{Z}/m$ where $m$ is prime to $p$ \cite{MR2254623}, and show that all $D_p$-covers (where $p \neq 2$) lift.

In this paper, we generalize the notion of Hurwitz tree to the case where $R$ is a complete discrete valuation ring of equal characteristic, specifically, $R=k[[t]]$, and $G \cong \mathbb{Z}/p$. The equal characteristic degeneration of $\mathbb{Z}/p$-covers was well-studied by Maugeais and Sa{\"i}di in \cite{MR2015076} and \cite{MR2377173}. Using their results together with Henrio's idea in \cite{2000math.....11098H}, we show that the existence of a flat deformation between Artin-Schreier covers of given branching data equates to the existence of a Hurwitz tree that satisfies certain criteria that are imposed by these data. Below is the main result of this paper. 

\begin{theorem}
\label{theoremmain}

Suppose $d, d_1, \ldots, d_r \not \equiv 0 \pmod{p}$ are integers such that $\sum_{i=1}^r (d_i+1)=d+1$. Fix a projective line $\mathbb{P}^1_{k[[t]]}$ and a $k$-point $b$ of $\mathbb{P}^1_{k[[t]]}$. Then there exists a $\mathbb{Z}/p$-Galois cover of $\mathbb{P}^1_{k[[t]]}$ whose special fiber is an Artin-Schreier cover over $k$, branched only at $b$ with ramification jump $d$ and whose generic fiber is an Artin-Schreier cover of $\mathbb{P}^1_{k((t))}$ branched at $r$ points that specialize to $b$ and which have ramification jumps $d_1, d_2, \ldots, d_r$ if and only if there exists a Hurwitz tree of type $\{d_1+1, d_2+1, \ldots, d_n+1\}$.
\end{theorem}


\textit{Type}, which will be defined in Definition \ref{defhurwitztree}, is a combinatorial invariant of a Hurwitz tree. Applying the Hurwitz tree criteria, we improve Theorem \ref{thmconnected} as below.

\begin{theorem}
\label{improveconnectedness}
\begin{itemize}
    \item When $p=5$, $\mathcal{AS}_g$  is connected if and only if $g \ge 14$ or $g=0,2$.
    \item  When $p>5$, $\mathcal{AS}_g$ is disconnected if $(p-1)/2 < g \le (p-1)(p-2)$.
\end{itemize}
\end{theorem}

\begin{remark}
The cases that remain unknown are when $p>5$ and $(p-1)(p-2)<g \le (p^3-2p^2+p-8)(p-1)/8$. Note that when $p$ and $g$ are fixed, one can apply Theorem \ref{theoremmain} and the computation strategy from Section \ref{secgeometryASg} to obtain a full pictures of how the irreducible components of $\mathcal{AS}_g$ fit together, hence also determine the connectedness of $\mathcal{AS}_g$. Our current program and computers cannot handle the calculations when $p$ and $g$ are large, though.  
\end{remark}

\subsection{Structure}

Section \ref{sectAS} gives a quick overview of Artin-Schreier theory, the moduli space of Artin-Schreier covers of fixed genus, and how to partition the space by the branching data of the points. In the same section, we link the geometry of $\mathcal{AS}_g$ with the deformation of $\mathbb{Z}/p$-covers (of genus $g$), and reduce Theorem \ref{theoremmain} to a local one (Theorem \ref{theoremmainlocal}). In \S \ref{seccoverdisc}, we examine the degeneration of {\'e}tale $\mathbb{Z}/p$-torsors on a disc using the language of Kato's refined Swan conductors. They are crucial to our approach, and distinguish this paper from \cite{DANG2020398}. Section \ref{sectionHurwitz} introduces the notion of Hurwitz tree and describes how to derive such a tree from a given $\mathbb{Z}/p$-deformation (\S \ref{seccovertotree}), thus proving the forward direction of Theorem \ref{theoremmain}. Section \ref{sechurwitztocover} considers the inverse process of \S \ref{seccovertotree}, constructing a $\mathbb{Z}/p$-deformation from a given Hurwitz tree, hence completing the proof of Theorem \ref{theoremmain}. We then recover some known results about deformations using the new technique in \S \ref{secconnection}. Finally, the proof of Theorem \ref{improveconnectedness} is given in \S \ref{secproofofimproveconnectedness}.

\begin{acknowledgements}
\label{secacknowledge}
The author would like to express his great gratitude to his advisor, Andrew Obus, for his guidance and very careful reading of the draft of this paper. He thanks Florian Pop for useful suggestions and proofreading an earlier version of this paper. He thanks Gjergji Zaimi, Fedor Petrov, Felipe Voloch, and many people who joined the discussion on this thread \url{https://mathoverflow.net/questions/310575/residues-of-frac1-prod-i-1n-x-p-ie-i}. Finally, he thanks the anonymous referees for their careful reading and their useful comments. This work is funded by the Simons Foudnation Grant Targeted for Institute of Mathematics, Vietnam Academy of Science and Technology, and NSF-DMS grants 1602054 and 1900396. 
\end{acknowledgements}

\section{Artin-Schreier Covers}
\label{sectAS}
\subsection{Artin-Schreier Theory}
Let $K$ be a field of characteristic $p>0$. If $L$ is a separable extension of $K$ of degree $p$, then the classical Artin-Schreier theory says that $L=K(\alpha)$ where $\alpha$ is a root of a polynomial equation $y^p-y=a$ for some $a \in K$ (see \cite[Section VI.6]{MR1878556}).

Let $Y$ be an Artin-Schreier $k$-curve. Then, by definition, there is a $\mathbb{Z}/p$-cover $\phi$: $Y \rightarrow \mathbb{P}_k^1$ with an affine equation of the form 
\begin{equation}
\label{eqnAS}
    y^p-y=f(x)
\end{equation}
\noindent for some non-constant rational function $f(x)\in k(x)$. Equation (\ref{eqnAS}) is called an \textit{Artin-Schreier equation}. Furthermore, a cover $\phi': Y' \rightarrow \mathbb{P}_k^1$ defined by $y^p-y=g(x)$ is isomorphic to $\phi$ if and only if $g(x)=f(x)+l(x)^p-l(x)$ for some $h(x) \in K(x)$. We call the map $\wp$ sending $l(x) \mapsto l(x)^p-l(x)$ the \textit{Artin-Schreier operator}. At each ramification point, there is a filtration of \textit{ramification groups in upper numbering} \cite[IV]{MR554237}. In our case, as the inertia group is $\mathbb{Z}/p$, the filtration has only one jump, which we call the \textit{ramification jump} at the corresponding branch point.  

\begin{example}
\label{exAS}
Suppose $p=5$. The cover of $Y$ of $\mathbb{P}^1_K$ defined by
\begin{equation}
\label{eqnASex}
    y^5-y=\frac{1}{x^5}+\frac{1}{(x-1)^2},
\end{equation}
\noindent is an Artin-Schreier curve. Note that the term $1/x^5$ is a $5$th-power. Hence, by the above discussion, one may add $\wp((-1/x)=(-1/x)^5-(-1/x)$ to the right-hand-side of (\ref{eqnASex}). The result is an Artin-Schreier equation of the form
\begin{equation}
\label{eqnASexreduced}
    y^5-y=-\frac{1}{x}+\frac{1}{(x-1)^2}.
\end{equation}
\end{example}
\begin{remark}
We say the Artin-Schreier equation (\ref{eqnASexreduced}) has \textit{reduced form}. That means the partial fraction decomposition of the right-hand-side of the equation only consists of terms of prime-to-$p$ degree. When $k$ is algebraically closed, any Artin-Schreier cover can be represented by a rational function $f(x)$ of reduced form. 
\end{remark}

 The rational function $f(x)$ in (\ref{eqnAS}) tells us everything about the ramification data of the cover $\phi$. Suppose $f(x)$ has $r$ poles $\mathcal{B}:=\{B_1,\ldots,B_{r}\}$ on $\mathbb{P}^1_k$. Let $d_j$ be the order of the pole of $f(x)$ at $B_j$. One may assume that $f(x)$ is in reduced form. Hence, the number $d_j$ is prime to $p$. It is an easy exercise to show that $\mathcal{B}$ is the collection of branch points of $\phi$, and $d_j$ is the ramification jump at $B_j$. We call $h_j:=d_j+1$ the \textit{conductor} at $B_j$. Then $h_j \ge 2$ and $h_j \not\equiv 1 \bmod p$. Moreover, the ramification divisor of $\phi$ is
\begin{equation}
\label{eqnramificationdivisorASW}
    D:=\sum_{j=1}^{r} (p-1)h_j Q_j
\end{equation} 
where $Q_j$ is the ramification point above $B_j$ (\cite{MR554237}, IV, Proposition 4). Hence, the degree of the different is $\sum_{j=1}^r(p-1)h_j$. Applying the Riemann-Hurwitz formula (\cite{MR0463157}, IV, Corollary 2.4), we obtain the following lemma. 

\begin{lemma} [{\cite[Lemma 2.6]{MR2985514}}]
\label{lemmagenus}
The genus of $Y$ is $g_Y=((\sum_{j=1}^{r}h_j)-2)(p-1)/2$. 
\end{lemma}

\begin{definition}
\label{defbranchingdatum}
 We call the $r \times 1$ matrix $[h_1, \ldots, h_r]^{\top}$ the \textit{branching datum} of $\phi$. For instance, the cover in Example \ref{exAS} has branching datum $[2,3]^{\top}$. Throughout the paper, we define $d$ by $g=d(p-1)/2$.  So, we have the identities:
$$d=\big(\sum_{j=1}^{r}h_j\big)-2 \text{ and } \sum_{j=1}^{r}h_j=d+2=2g/(p-1)+2.$$ 
\end{definition}

\begin{remark}
\label{remarksumofconductors}
The above lemma shows that all the Artin-Schreier $k$-curves with the same genus $g$ have the same sum of conductors $d+2$. That is the essential difference between an Artin-Schreier cover and a $\mathbb{Z}/p$-cover in characteristic different from $p$. For each $\mathbb{Z}/p$-tamely-ramified-cover, a branch point contributes $p-1$ to the degree of its ramification divisor. We thus say a branch point of a wildly ramified cover contributes more to the degree of the different. Therefore, every $\mathbb{Z}/p$-tamely-ramified-cover of genus $g$ must have the same number of branch points. It hence makes sense to group Artin-Schreier covers of the same genus by their branching data. This idea is utilized by Pries and Zhu in \cite{MR2985514}, and will be discussed in the next section.
\end{remark}

\begin{remark}
\label{remarkdeformationASW}
When the Galois group is $\mathbb{Z}/p^n$, the Artin-Schreier theory is generalized by the Artin-Schreier-Witt theory \cite[\S 26]{MR2371763} \cite[\S 4]{MR1878556}. It says that a $\mathbb{Z}/p^n$-cover of $\mathbb{P}^1_k$ is determined by some certain length-$n$-Witt-vector over $K$. We will discuss the deformations of this family of covers in a forthcoming paper. 
\end{remark}

\subsection{Deformations of Artin-Schreier covers}


Suppose $C \xrightarrow{\phi} \mathbb{P}^1_k$ is a $G$-Galois cover over $k$, where $C$ is a smooth, projective, connected $k$-curve. Suppose, moreover, that $A$ is a Noetherian, complete $k$-algebra with residue field $k$. Let
\[ \Def_{\phi} : \Alg/k \xrightarrow{} \text{Set} \]
\noindent be the functor which to any $A, f: A \xrightarrow{} k \in \Alg/k$ associates classes of $G$-Galois covers $\mathscr{C} \xrightarrow{\psi} \mathbb{P}^1_A$ that make the following cartesian diagram commute 

\begin{equation}
\label{defndef}
     \begin{tikzcd}
C \arrow{d}{\phi} \arrow{r}{}
& \mathscr{C} \arrow{d}{\psi} \\
\mathbb{P}^1_k \arrow{r}[blue]{} \arrow[black]{d}{} & \mathbb{P}^1_A \arrow[black]{d}{}\\
\spec k \arrow[black]{r}{f} & \spec A,
\end{tikzcd}
\end{equation}

\noindent and so that the $G$-action on $\mathscr{C}$ induces the original action on $C$. We say $\psi$ is a \textit{deformation} of $\phi$ over $A$, or $\phi$ is deformed (over $A$) to $\psi$. For more details, see \cite[\S 2]{MR1767273}. In this paper, we focus on the case where $G=\mathbb{Z}/p$ and $A=k[[t]]$.

\begin{remark}
One unique aspect of characteristic $p$ is that there exist flat deformations of a wildly ramified cover over rings of equal characteristic so that the number of branch points changes, but the genus does not. That gives us a way to investigate a cover in characteristic $p$: finding a connection of it with a slightly different one via equal characteristic deformation. For example, using some (equal characteristic) deformations of $\mathbb{Z}/p^n$-covers \cite[Lemma 3.2]{MR3194816}, Pop reduced the lifting problem for cyclic groups (a.k.a., the \textit{Oort conjecture}) to the case that had been solved by Obus and Wewers in \cite{MR3194815}. Another perk of studying these deformations is understanding the geometry of the moduli space that parameterizes Galois covers. That will be discussed further in \S \ref{secmoduli}.
\end{remark}

\begin{remark}
\label{remarktype}
Suppose $\phi: Y \xrightarrow{} \mathbb{P}^1_k$ is a $\mathbb{Z}/p$-cover that branched at $r$ points $\{P_1, \ldots, P_r\}$ with conductor $h_i$ at $P_i$. Suppose, moreover, that $\psi$ is a smooth deformation of $\phi$ over a complete discrete valuation ring of \textit{equal characteristic} whose generic fiber $\psi_{\eta}$ has branch locus $\{P_{1,1}, \ldots, P_{1,m_1}, \ldots, \allowbreak P_{r,1}, \allowbreak \ldots, P_{r,m_r} \}$, where each $P_{i,j}$ reduces to $P_i$ and has conductor $h_{i,j}$. Then it follows from the discussion in Remark \ref{remarksumofconductors} that $\sum_{i=1}^r\sum_{j=1}^{m_j} h_{i,j}=\sum_{i=1}^r h_i$. We say the deformation $\psi$ has \textit{type}
\[ [h_1, \ldots, h_r]^{\top} \xrightarrow{} [h_{1,1}, \ldots, h_{1,m_1}, h_{2,1}, \ldots, h_{r,1}, \ldots, h_{r,m_r}]^{\top}.  \]
\end{remark}

\subsection{The moduli space of Artin-Schreier covers and their deformations}
\label{secmoduli}
In \cite{MR2985514}, the authors introduce the moduli space of Artin-Schreier $k$-curves of genus $g$, which they denote by $\mathcal{AS}_g$. They enumerate a family of locally closed strata of $\mathcal{AS}_g$ by partitions of the integer $d+2$ such that no entries are congruent to $1$ modulo $p$, i.e., all the possible branching data of Artin-Schreier curves of genus $g$. We call the collection of those partitions $\Omega_{d+2}$. For instance, the partition $\overrightarrow{E}=\{h_1,\ldots,h_r\}$  of $d+2$ is associated with the stratum $\Gamma_{\overrightarrow{E}}$, which is the collection of all the points of $\mathcal{AS}_g$ that represent Artin-Schreier curves with branching datum $[h_1, \ldots, h_r]^{\top}$. We write $\overrightarrow{E}_1 \prec \overrightarrow{E}_2$ if the latter one is a refinement of the former one.

\begin{example}
Suppose $p=5$ and $g=14$. Then $d=7$, and the strata of $\mathcal{AS}_g$ correspond to the following partitions of $d+2$: $\{9\}, \{7,2\}, \{5,4\}, \{5,2,2\}, \{4,3,2\}, \{3,3,3\}$, and $\{3,2,2,2\}$.
\end{example}

The following result shows that one can relate the geometry of $\mathcal{AS}_g$ with the existence of equal characteristic deformations between curves in different strata (hence have distinct branching data). We say there \textit{always} exist deformations of type $[\overrightarrow{E}_1]^{\top} \xrightarrow{} [\overrightarrow{E}_2]^{\top}$ if, given an arbitrary Artin-Schreier curve of branching datum $[\overrightarrow{E}_1]^{\top}$, we can deform it to one with branching datum $[\overrightarrow{E}_2]^{\top}$.

\begin{proposition}[{\cite[Proposition 3.4]{DANG2020398}}]
\label{propdeformclosure}
Suppose $\overrightarrow{E}_1$ and $\overrightarrow{E}_2$ are two partitions of $d+2$. Then the stratum $\Gamma_{\overrightarrow{E}_1}$ is contained in the closure of $\Gamma_{\overrightarrow{E}_2}$ if and only if there exists a deformation over $k[[t]]$ from a point in $\Gamma_{\overrightarrow{E}_1}$ to one in $\Gamma_{\overrightarrow{E}_2}$. In particular, there always exist deformations of type $[\overrightarrow{E}_1]^{\top} \xrightarrow{} [\overrightarrow{E}_2]^{\top}$ if one can find a deformation of such type.
\end{proposition}

Furthermore, the next result shows that the closure of a stratum is simply a union of strata.

\begin{proposition}[{\cite[Corollary 3.6]{DANG2020398}}]
\label{propclosurestratum}
Let $\overrightarrow{E}_1$ and $\overrightarrow{E}_2$ be as in the previous proposition. Suppose $\Gamma_{\overrightarrow{E}_1}$ is not in the closure of $\Gamma_{\overrightarrow{E}_2}$. Then $\Gamma_{\overrightarrow{E}_1}$ is disjoint from the closure of $\Gamma_{\overrightarrow{E}_2}$.
\end{proposition}

\noindent Therefore, understanding these deformations can give us a full picture of the moduli space $\mathcal{AS}_g$. We thus want to answer the following question.

\begin{question}[Deformation of Artin-Schreier covers problem]
\label{questionglobal}
Suppose we are given $\overrightarrow{E}_1$ and $\overrightarrow{E}_2$ in $\Omega_{d+2}$. Does there exists a deformation over $k[[t]]$ of type $[\overrightarrow{E}_1]^{\top} \xrightarrow{} [\overrightarrow{E}_2]^{\top}$?
\end{question}


\subsubsection{The graph \texorpdfstring{$C_d$}{Cd}}

Let us fix a prime $p$ and construct a directed graph $C_d$ (where $d$ is as in Definition \ref{defbranchingdatum}). The vertices of the graph correspond to the partitions $\overrightarrow{E}$ in $\Omega_d$. There is an arrow from $\overrightarrow{E}$ to $\overrightarrow{E}'$ if and only if $\overrightarrow{E}\prec \overrightarrow{E}'$, and $\Gamma_{\overrightarrow{E}}$ lies in the closure $\Gamma_{\overrightarrow{E}'}$.

In general topology, if one irreducible subset of a space lies in the closure of another, then they are contained in the same connected component of that space. Thus, if $C_d$ is connected, then so is $\mathcal{AS}_g$. It is straightforward to check that the converse also holds. From now on, we will study the geometry of $C_d$.

\begin{example}
Suppose $p=5$ and $g=14$. Then $d=7$. Using the information from \cite[Theorem 3.10]{DANG2020398}, we draw a subgraph of $C_7$ in Figure \ref{figC7}. As the subgraph is connected, the graph $C_7$ is connected, and so is $\mathcal{AS}_{14}$.
\begin{figure}
    \centering
\begin{tikzpicture}[-latex ,auto ,node distance =1.1 cm and 1.7cm ,on grid ,
semithick] 
\node (C)
{$\{9\}$};
\node (A) [below left=of C] {$\{5,4\}$};
\node (B) [below right =of C] {$\{7,2\}$};
\node (D) [below =of C] {$\{3,3,3\}$} ;
\node (E) [below =of A] {$\{4,3,2\}$} ;
\node (F) [below =of B] {$\{5,2,2\}$} ;
\node (G) [below =of E] {$\{3,2,2,2\}$} ;
\path (C) edge [bend right =15] node[below =0.15 cm] { } (A);
\path (C) edge node[below =0.15 cm] { } (D);
\path (B) edge [bend left =15] node[below =0.15 cm] { } (E);
\path (B) edge node[below =0.15 cm] { } (F);
\path (C) edge [bend right =90] node[below =0.15 cm] { } (G);
\path (C) edge [bend left =5] node[below =0.15 cm] { } (E);
\end{tikzpicture}
    \caption{A subgraph of $C_7$}
    \label{figC7}
\end{figure}
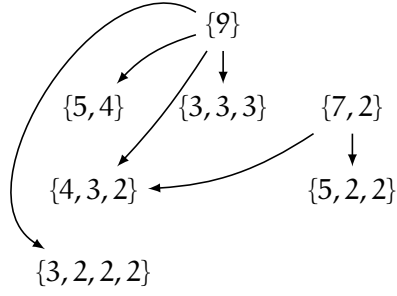
\end{example}

\begin{remark}
One can apply the results later in this paper (Proposition \ref{propchaindiffforms} and Proposition \ref{propgrobner}) to show that the above diagram is the complete $C_7$. Thus, we can read from the graph that the irreducible components of $\mathcal{AS}_{14}$ are the closures of the following strata: $\Gamma_{\{5,4\}}$,$\Gamma_{\{3,3,3\}}$, $\Gamma_{\{4,3,2\}}$, $\Gamma_{\{5,2,2\}}$, and $\Gamma_{\{3,2,2,2\}}$. Furthermore, the intersection of $\overline{\Gamma}_{\{3,3,3\}}$ and $\overline{\Gamma}_{\{3,2,2,2\}}$ is $\Gamma_{\{9\}}$, and the intersection of $\overline{\Gamma}_{\{4,3,2\}}$ and $\overline{\Gamma}_{\{5,2,2\}}$ is $\Gamma_{\{7,2\}}$ ($\overline{\Gamma}$ denotes the closure of the strata $\Gamma$).
\end{remark}

\subsection{Reduction to the local deformation problem}
\label{secreducetolocal}

We first state a fact about germs of Artin-Schreier curves, which implies that, locally, they are easy to control.

\begin{proposition}
\label{localAS}
Suppose $\phi: Y \rightarrow \mathbb{P}^1_k$ is an Artin-Schreier cover and $P \in \mathbb{P}^1_k$ is a branch point of $\phi$. Then the localization of $\phi$ at $P$ is determined by the ramification jump at $P$. 
\end{proposition}

\begin{proof}

See Lemma $2.1.2$ of \cite{MR2016596}.
\end{proof}
That means a $\mathbb{Z}/p$-cover of $\spec k[[x]]$ of conductor $h$, up to a change of the variable $x$, is isomorphic to one defined by
\[ y^p-y =\frac{1}{x^{h-1}}. \]
The following local-global principle type result will help us to reduce our study of Artin-Schreier deformations to the local case.
 
\begin{proposition}[{c.f. \cite[Proposition 3.4]{DANG2020398}}]
\label{propreduce}
Suppose $\overrightarrow{E}_1=\{h_1,h_2,\ldots,h_n\}$ and $\overrightarrow{E}_2$ are in $\Omega_{h}$. Then there exists a deformation of type $[\overrightarrow{E}_1]^{\top} \xrightarrow{} [\overrightarrow{E}_2]^{\top}$ if and only if there exist $n$ partitions $\{h_i\}\prec \overrightarrow{E}_{i,2}:=\{h_{i,1},\ldots, h_{i,m_i}\} \subseteq \overrightarrow{E}_2$ ($1 \le i \le n$), where the $\overrightarrow{E}_{i,2}$'s partition $\overrightarrow{E}_2$, and $n$ deformations $\psi_1, \psi_2 \ldots, \psi_n$ over $k[[t]]$, where $\psi_i$ has type $[h_i] \xrightarrow{} [h_{i,1},\ldots, h_{i,m_i}]^{\top}$. In particular, a deformation of this type exists only if $\overrightarrow{E}_1 \prec \overrightarrow{E}_2$.
\end{proposition}

\begin{example}
Suppose $p=5$. Then, by Proposition \ref{propreduce}, there exists a deformation (over $k[[t]]$) of type $[7,4]^{\top} \xrightarrow{} [4,3,2,2]^{\top}$ if and only if there exist either two deformations of types $[7] \xrightarrow{} [3,2,2]^{\top}$ and $[4] \xrightarrow{} [4]$ (trivial deformation), or two deformations of types $[7] \xrightarrow{} [4,3]^{\top}$ and $[4] \xrightarrow{} [2,2]^{\top}$.
\end{example}

\begin{remark}
It follows immediately from the above proposition that, if we have a deformation as in Remark \ref{remarktype}, then $\sum_{j=1}^{m_i} h_{i,j}=h_i$ for all $i$.
\end{remark}

Proposition \ref{localAS} and Proposition \ref{propreduce} suggest that we may assume $\phi$ is a $\mathbb{Z}/p$-extension $k[[z]]/k[[x]]$ branched only  at $x=0$ with conductor $h$. Therefore, we may also think of a deformation $\phi$ over $k[[t]]$ as a $\mathbb{Z}/p$-cover $k[[t]][[Z]]/k[[t]][[X]]$. We thus reduce Question \ref{questionglobal} to the following.

\begin{question}[Local deformation of Artin-Schreier covers]
\label{questionlocal}

Suppose $\{h\} \prec \{h_1, \ldots, h_r\} \in \Omega_h$ and $\phi$ is a $\mathbb{Z}/p$-cover $k[[z]]/k[[x]]$ given by $y^p-y=\frac{1}{x^{h-1}}$. Define $R:=k[[t]]$. Does there exist a deformation $R[[Z]]/R[[X]]$ of $\phi$ with generic branching datum $[h_1, \ldots, h_r]^{\top}$?

\end{question}

Hence, the question can be fully answered by the following local version of Theorem \ref{theoremmain}.

\begin{theorem}
\label{theoremmainlocal}
 Let $\phi: k[[z]] \xrightarrow{} k[[x]]$ be a local $G$-cover with conductor $d$. Then there exists a deformation of $\phi$ over $k[[t]]$ of type $[d+1] \xrightarrow{} [d_1+1, \ldots, d_r+1]^{\top}$ if and only if there exists a Hurwitz tree of type $\{d_1+1, \ldots, d_r+1\}$.
\end{theorem}

\subsection{Birational deformation and the different criterion}

Usually, when dealing with Galois extensions of $k[[x]]$, it will be more convenient to deal with extensions of fraction fields than extensions of rings. So we will often want to think of a Galois ring extension in terms of the associated extension of fraction fields. 

\begin{definition}
\label{defbirationaldeformatiton}
Suppose $A/k[[x]]$ is a local $G$-extension. Suppose, moreover, that $M / \Frac (R[[X]])$ where $R/k[[t]]$ finite, is a $G$-extension, and $A_R$ is the integral closure of $R[[X]]$ in $M$. We say $M / \Frac (R[[X]])$ is a \textit{birational deformation} of $A/k[[x]]$ if
\begin{enumerate}
    \item The integral closure of $A_R \otimes_R k$ is isomorphic to $A$, and
    \item The $G$-action on $\Frac (A)= \Frac (A_R \otimes_R k)$ induced from that on $A_R$ restricts to the given $G$-action on $A$.
\end{enumerate}
\end{definition}

The following criterion is extremely useful for seeing when a birational deformation is actually a deformation (i.e., when $A_R \otimes_R k$ is already integrally closed, this isomorphic to $A$).

\begin{proposition}[(The different criterion) {\cite[I,3.4]{MR1645000}}]
\label{propdifferentcriterion} Suppose $A_R/R[[X]]$ is a birational deformation of the $G$-Galois extension $A/k[[x]]$. Let $K=\Frac R$, let $\delta_{\eta}$ be the degree of the different of $(A_R \otimes_R K)(R[[X]] \otimes_R K)$, and let $\delta_s$ be the degree of the different of $A/k[[x]]$. Then $\delta_s \le \delta_{\eta}$, and equality holds if and only if $A_R/R[[X]]$ is a deformation of $A/k[[x]]$.

\end{proposition}

\begin{example}
\label{exfirst}
Let $p=5$, $g=10$, hence, $d=5$. Consider the $\mathbb{Z}/p$-cover $\psi$ given by the normalization of $\mathbb{P}^1_{k[[t]]}$ over the extension of its fraction field defined by the following equation:
\begin{equation}
\label{eqnex1}
    Y^5-Y=\dfrac{-2X+t^5}{(-2)X^5(X-t^5)^2}=H(X,t).
\end{equation}
The special fiber is birational to the $\mathbb{Z}/5$-cover
\[y^5-y=\dfrac{1}{x^{6}},\]
\noindent which is branched at $0$ with conductor $7$.

On the generic fiber, when $t \neq 0$, the partial fraction decomposition of $H(X,t)$ is of the form:
\[\dfrac{-1}{2t^5X^5}+\dfrac{1}{2t^{15}X^3}+\dfrac{1}{t^{20}X^2}+\dfrac{3}{2t^{25}X}+\dfrac{1}{2t^{20}(X-t^5)^2}-\dfrac{3}{2t^{25}(X-t^5)}.\]
\noindent Adding $\frac{1}{2t^5X^5}-\frac{1}{2tX}$ to $H(X,t)$, we get
\[\dfrac{-1}{2tX}+\dfrac{1}{2t^{15}X^3}+\dfrac{1}{t^{20}X^2}+\dfrac{3}{2t^{25}X}+\dfrac{1}{2t^{20}(X-t^5)^2}-\dfrac{3}{2t^{25}(X-t^5)},\]
\noindent which is a reduced form. Hence, the generic fiber is branched at two points $X=0$ and $X=t$, which have conductors $4$ and $3$, respectively. It then follows from Proposition \ref{propdifferentcriterion} that $\psi$ is a deformation of type $[7] \xrightarrow{} [4,3]^{\top}$.

\end{example}

\begin{example}
\label{examplemain}
One can check, in the same fashion as above, that the $\mathbb{Z}/5$-cover given by
\[ Y^5-Y=\frac{-2X+t^{10}}{(-2)X^5(X-t^{10})^2(X-t^5)^5} \]
\noindent is a deformation of $y^5-y=\frac{1}{x^{11}}$ over $k[[t]]$, where the generic branch points $0, t^{10}$, and $t^{5}$ have conductors $4, 3$, and $5$, respectively. Hence, it is a deformation of type $[12] \xrightarrow{} [4,3,5]^{\top}$.
\end{example}

\begin{remark}
Suppose $\phi$ is a Galois cover over $k$. We can also apply the criterion to determine whether a deformation $\psi$ of $\phi$ over a ring $R$ of mixed characteristic (e.g., $R$ is a finite extension of $W(k)$, where $W(k)$ is a Witt vector over $k$) is smooth. The lifting problem for Galois covers concerns the existence of a Galois cover in characteristic $0$ that reduces to a given one in characteristic $p$. See \cite[\S 4]{MR3194815} for some explicit lifts of local Artin-Schreier covers and $\mathbb{Z}/p^2$-covers. Some good expositions of the problem are \cite{MR3051249}, \cite{MR2254623}, \cite{MR3194815}, and \cite{MR3874854}. Recall that the notion of Hurwitz tree, which we will discuss in \S \ref{sectionHurwitz}, is first introduced in \cite{2000math.....11098H} to tackle the lifting problem for Artin-Schreier covers. Theorem \ref{theoremmain} is the exact equal characteristic analog of \cite[Th{\'e}or{\`e}m de r{\'e}alisation]{2000math.....11098H}.
\end{remark}

\section{Degeneration of \texorpdfstring{$\mathbb{Z}/p$}{Zp}-covers}
\label{seccoverdisc}

In this section, we study the degeneration of $\mathbb{Z}/p$-covers of $\spec R[[X]]$, where $R$ is a complete discrete valuation of equal characteristic and with uniformizer $\pi$. Normalize the canonical valuation on $R$ so that $\nu(\pi)=1$. Set $K:=\Frac R$ and $\mathbb{K}:=\Frac R[[X]]$. We first introduce a geometric interpretation of $\spec R[[X]]$ using the language of non-archimedean geometry.   

\subsection{Discs and annuli}
\label{secdiscannuli}
Firstly, we identify the $K$-analytic points of $\spec R[[X]]$ with
\[D= \{ u \in (\mathbb{A}^1_K)^{\an} \mid \nu(u)>0  \} \]
\noindent by plugging in $X=u$. We call $X$ a \textit{parameter} for the \textit{open unit disc} $D$ \textit{with center} $0$.

Let $R\{X\} \subsetneq R [[X]]$ consist of power series for which the coefficients tend to $0$, i.e., 
$$R\{X\}=\bigg\{ \sum_{i \ge m} a_iX^i \mid \lim_{i \to \infty} \nu (a_i) =\infty \bigg\} $$
\noindent As before, $K$-points of a \textit{closed unit disc} $\spec R\{X\}$ can be identified with
\[ \mathcal{D}=\{ u \in (\mathbb{A}^1_K)^{\an} \mid \nu(u)\ge 0 \}. \]
The \textit{boundary} of the open (or closed) unit disc is represented by the scheme $\spec R[[X^{-1}]]\{X\}$. Note that the ring $S:=R[[X^{-1}]]\{X\}$ is a complete discrete valuation ring with residue field $\overline{S}=k((x))$, uniformizing element $\pi$ and fraction field $S \otimes_R K$.

For $r\in \mathbb{Q}_{\ge 0}$, and $a \in K$ such that $\lvert a \rvert=r$, the open (resp. closed) disc of radius $r$ is characterized by the scheme $\spec R[[a^{-1}X]]$ (resp. $\spec R\{a^{-1}X\}$). Its set of $K$-points is isomorphic to the open (resp. closed) unit disc under the map $X \mapsto a^{-1}X$. Denote by 
$$\mathcal{D}[s,z]:=\{  u \in (\mathbb{A}^1_K)^{\an} \mid \nu(u-z)\ge s \},$$
\noindent where $z \in (\mathbb{A}^1_K)^{\an}$, the closed disc of radius $p^{-s}$ centered at $X-z$ , and $\mathcal{D}[s]:=\mathcal{D}[s,0]$. One can associate with $\mathcal{D}[s,z]$ the "Gauss valuation" $\nu_{s,z}$ that is defined by
\begin{equation*}
    \nu_{s,z}(f)=\inf_{a \in \mathcal{D}[s,z]} (\nu(f(a)),
\end{equation*}
for each $f \in \mathbb{K}^{\times}$. This is a discrete valuation on $\mathbb{K}$ which extends the valuation $\nu$ on $K$, and has the property $\nu_{s,z}(X-z)=s$. We denote by $\kappa_s$ the residue field of $\mathbb{K}$ with respect to the valuation $\nu_{s,z}$. That is the function field of the canonical reduction $\overline{\mathcal{D}}[s,z]$ of $\mathcal{D}[s,z]$. In fact, $\overline{\mathcal{D}}[s,z]$ is isomorphic to the affine line over $k$ with function field $\kappa_{s,z}=k(x_{s,z})$, where $x_{s,z}$ is the image of $\pi^{-s}(X-z)$ in $\kappa_{s,z}$. For a closed point $\overline{x} \in \overline{\mathcal{D}}[s,z]$, we let $\ord_{\overline{x}}: \kappa_{s,z}^{\times} \xrightarrow{} \mathbb{Z}$ denote the normalized discrete valuation corresponding to the specialization of $\overline{x}$ on $\overline{\mathcal{D}}[s,z]$. We let $\ord_{\infty}$ denote the unique normalized discrete valuation on $\kappa_{s,z}$ corresponding to the "point at infinity". 

The \textit{open annulus of thickness $\epsilon$} is described by $\spec R[[X,U]]/(XU-a)$, where $a \in K$ is such that $\nu(a)=\epsilon$. The open annulus has two boundaries, one given by $\spec R[[X]]\{X^{-1}\}$ and one given by $\spec R[[U]]\{U^{-1}\}$. Note that two annuli over $R$ are isomorphic if and only if they have the same thickness.

For $F \in \mathbb{K}$, $z \in (\mathbb{A}^1_K)^{\an}$, and $s \in \mathbb{Q}_{\ge 0}$, we let $[F]_{s,z}$ denotes the image of $\pi^{-\nu_{s,z}(F)}F$ in the residue field $\kappa_{s,z}$.

\subsection{Semistable model and a partition of a disc}
\label{secsemistablemodel}

Consider the open unit disc $D:=\spec R[[X]]$, and suppose we are given $x_{1,K}, \ldots, x_{r,K}$ in $D(K)$, with $r \ge 2$. We can think of $x_{1,K}, \ldots, x_{r,K}$ as elements of the maximal ideal of $R$. Let $D^{\st}$ be a modification of $D$ such that
\begin{itemize}
    \item the exceptional divisor $\overline{D}$ of the blow-up $D^{\st} \xrightarrow{} \spec R[[X]]$ is a semistable curve over $k$,
    \item the fixed points $x_{b,K}$ specialize to pairwise distinct smooth points $x_b$ on $\overline{D}$, and
    \item if $\overline{x}_0$ denotes the unique point on $\overline{D}$ which lies in the closure of $D^{\st} \otimes k \setminus \overline{D}$,
\end{itemize}
then $(\overline{D};\overline{x}_0, (x_b))$ is stably marked. We call $D^{\st}$ the \emph{semi-stable model of the marked disc $(D; x_1, \ldots, \allowbreak x_r)$}, and $(\overline{D};\overline{x}_0, (x_b))$ its \textit{special fiber}. The dual graph of the special fiber is a tree whose \textit{leaves} correspond to the marked points and whose \textit{root} corresponds to $\overline{x}_0$.

\begin{example}
\label{examplesemistablemodel}
Suppose we are given a $\mathbb{Z}/5$-cover (in characteristic $5$) of $\mathbb{P}^1_K$ that has four branch points $X=0$, $X=t^5$, $X=t^5(1+t^5)$, and $X=t^{10}$. The left graph of Figure \ref{figspecialfiberdual} represents a semi-stable model of the open unit disc $D=\spec R[[X]]$ marked by the branch points of the cover. The tree on the right is its dual graph.

We associate with each edge an annulus. In this example, $e_0$ corresponds to the spectrum of $R[[X,X_1]]/(XX_1-t^5)$ and $e_1$ resembles $\spec R[[X_1^{-1}-1,V]](V(X_1^{-1}-1)-t^5)$. We say $e_0$ has \textit{thickness} $1$, which is the thickness of the associated annulus divided by $p=5$. Moreover, each vertex, which is not the root of the tree, is associated with a \textit{punctured disc}. For instance, the vertex $v_1$ corresponds to $\spec R\{X_1^{-1},X_1,(X_1^{-1}-1)^{-1}\}$. That can also be thought of as the complement of the open discs $\spec R[[X_1]]$ and $\spec R[[(X_1^{-1}-1)^{-1}]]$ inside the closed disc $\spec R\{X_1^{-1}\}$. The root is linked with $\spec R\{X\}$. Finally, each leaf is associated with an open disc. To illustrate, one correlated with $\overline{t^5}$ is $\spec R[[t^{-5}(X-t^5)]]$. Table \ref{tabpartitions} shows where the $K$-points of the disc specialize. 

Note that, on the punctured disc corresponding to $v_1$, the marked points $t^5$ and $t^5+t^{10}$ (resp. $0$ and $t^{10}$) specialize to $\overline{1}$ (resp. to $\overline{0})$. We call the directions on the dual graph from $v_1$ toward $e_2, e_1$, and $e_0$ the \textit{directions} corresponding to (or with respect to) $\overline{0},\overline{1}$, and $\overline{\infty}$ , respectively.

\begin{center}
\begin{figure}[ht]
\centering
\begin{tikzpicture}[line cap=round,line join=round,>=triangle 45,x=1cm,y=1cm]
\clip(-5.913951741732328,-0.42740404202494114) rectangle (10.236822781996308,4.043296939568755);
\draw [line width=2pt] (1,3.5)-- (1,0);
\draw [line width=2pt] (-4.900693696805673,3.5)-- (-4.900693696805673,2.5);
\draw [line width=2pt] (0.5,2)-- (1.52,2);
\draw [line width=2pt] (0.5,1.2)-- (1.5,1.2);
\draw (-4.846846996414543,2.92063948653854) node[anchor=north west] {$\overline{\infty}$};
\draw (1.901869501000637,2.2829169096693054) node[anchor=north west] {$\overline{0}$};
\draw (1.8009271602273331,1.518978406128035) node[anchor=north west] {$\overline{t^{10}}$};
\draw (1.0510697716256465,0.8812558292588005) node[anchor=north west] {$\overline{D}_3$};
\draw [line width=2pt] (2.497760136302129,3)-- (-5.502239863697875,3);
\draw [line width=2pt] (-3.1869684357952335,3.461122757720566)-- (-3.1869684357952335,-0.03887724227943412);
\draw [line width=2pt] (-3.686968435795234,1.9611227577205659)-- (-2.666968435795234,1.9611227577205659);
\draw [line width=2pt] (-3.686968435795234,1.1611227577205658)-- (-2.686968435795234,1.1611227577205658);
\draw (-2.366549480270502,2.1832727570334876) node[anchor=north west] {$\overline{t^5}$};
\draw (-2.45307148664762,1.419334253492217) node[anchor=north west] {$\overline{t^5+t^{10}}$};
\draw (-2.856840849740836,0.8214693376773098) node[anchor=north west] {$\overline{D}_2$};
\draw (-1.4580684133107664,3.7975080297337374) node[anchor=north west] {$\overline{D}_1$};
\draw [line width=2pt] (5.5,1.5)-- (3.879577221035335,2.587879622356322);
\draw [line width=2pt] (5.5,1.5)-- (6.879577221035335,2.587879622356322);
\draw [line width=2pt] (5.5,1.5)-- (5.5,0);
\draw [line width=2pt] (3.879577221035335,2.587879622356322)-- (3.379577221035335,3.58787962235632);
\draw [line width=2pt] (3.879577221035335,2.587879622356322)-- (4.379577221035335,3.58787962235632);
\draw [line width=2pt] (6.879577221035335,2.587879622356322)-- (6.379577221035335,3.58787962235632);
\draw [line width=2pt] (6.879577221035335,2.587879622356322)-- (7.379577221035339,3.58787962235632);
\begin{scriptsize}
\draw [fill=black] (5.5,1.5) circle (2.5pt);
\draw[color=black] (6.15586814787559,1.4890851603372897) node {$v_1$};
\draw [fill=black] (3.879577221035335,2.587879622356322) circle (2.5pt);
\draw[color=black] (3.3439029406192655,2.565242008804123) node {$v_2$};
\draw[color=black] (4.944560058595942,2.3061672119509966) node {$e_1$};
\draw [fill=black] (6.879577221035335,2.587879622356322) circle (2.5pt);
\draw[color=black] (7.367176237155237,2.631671443894668) node {$v_3$};
\draw[color=black] (5.896302128744237,2.2596666073876146) node {$e_2$};
\draw [fill=black] (5.5,0) circle (2.5pt);
\draw[color=black] (6.213549485460335,0.0010658143090758049) node {$v_0$};
\draw[color=black] (5.146444740142551,0.7450754873231826) node {$e_0$};
\draw [fill=black] (3.379577221035335,3.58787962235632) circle (2.5pt);
\draw[color=black] (3.7620926381086672,3.727757122888665) node {$[\overline{t^5}]$};
\draw [fill=black] (4.379577221035335,3.58787962235632) circle (2.5pt);
\draw[color=black] (5.002241396180688,3.727757122888665) node {$[\overline{t^5+t^{10}}]$};
\draw [fill=black] (6.379577221035335,3.58787962235632) circle (2.5pt);
\draw[color=black] (6.920145870873463,3.727757122888665) node {$[\overline{0}]$};
\draw [fill=black] (7.379577221035339,3.58787962235632) circle (2.5pt);
\draw[color=black] (8.001670950587435,3.7211141793796103) node {$[\overline{t^{10}}]$};
\end{scriptsize}
\end{tikzpicture}
    \caption{The special fiber $\overline{D}$ and its dual graph}
    \label{figspecialfiberdual}
\end{figure}
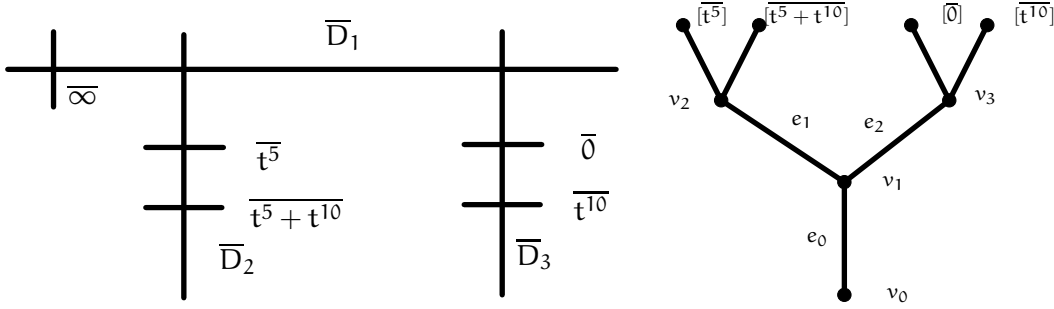
\end{center}

\begin{table}[ht]
\small
    \centering
\begin{tabular}{ |p{3.4cm}|p{5.8cm}|p{5.2cm}|  }
\hline
 Subscheme $\overline{V}$ of $\overline{D}$ & Points of $D(K)$ that specialize to $\overline{V}$ & Associated algebraic object \\
 \hline
 $\overline{\infty}$ & $ \{Y \mid 0< \nu(Y)<5  \}$  & $ R[[X,X_1]]/(XX_1-t^5)$ \\
 
   $\overline{D}_3 \cap \overline{D}_1$ & $\{Y \mid 5  < \nu(Y) < 10\}$ & $ R[[X_1^{-1},X_2]]/(X_1^{-1}X_2-t^5)$ \\
 
 $\overline{D}_3 \setminus \overline{D}_1$ & $\{Y \mid \nu(Y) \ge 10 \}$ & $ R\{X_2^{-1}\}$   \\
 
   $\overline{D}_2 \cap \overline{D}_1$ & $\{Y \mid 5 < \nu(Y-t^5) <10 \}$ & $ R[[X_1^{-1}-1,V]](V(X_1^{-1}-1)-t^5)$ \\
  
  $\overline{D}_2 \setminus \overline{D}_1$ & $\{Y \mid \nu(Y-t^5) \ge 10 \wedge \nu(Y)=5 \}$  & $ R\{V^{-1}\}$  \\ 

     $\overline{D}_1 \setminus ( \overline{D}_3 \cup \overline{D}_2 \cup \{\overline{\infty}\})$ & $\{Y \mid \nu(Y)=5 \wedge \nu(Y-t^5)=5 \}$ & $ R\{X_1^{-1},X_1,(X_1^{-1}-1)^{-1}\}$ \\
     \hline
\end{tabular}
\hspace{5mm}
    \caption{Partitions of $D(K)$}
    \label{tabpartitions}
\end{table}
\end{example}

\subsection{Reduction of covers}

Let $\psi: R[[Z]] \xrightarrow{} R[[X]]$ be a Galois cover. After enlarging our ground field $K$, we may assume that $\psi$ is \textit{weakly unramified} with respect to the Gauss valuation $\nu_0$ (that maps $X$ to $0$), see \cite{MR0321929}. By definition, this means that for all extensions $w$ of $\nu_0$ to the function field of $R[[Z]]$, the ramification index $e(w/\nu_0)$ is equal to $1$. It then follows that the special fiber $\spec R[[Z]] \otimes_R k$ is reduced.

\begin{definition}
We say that the cover $\psi$ has \textit{{\'e}tale reduction} if the induced map $\spec R[[Z]] \otimes_R k \xrightarrow{\phi} \spec R[[X]] \otimes_R k$ is generically {\'e}tale.
\end{definition}

\begin{definition}
If $\psi$ has {\'e}tale reduction, we call $\phi$ the \textit{reduction} of $\psi$, and $\psi$ a \textit{deformation} of $\phi$ over $R$. We say $\psi$ has \textit{good reduction} if it has {\'e}tale reduction, and $\phi$ is smooth.
\end{definition}

Recall Proposition \ref{localAS} shows that a local Artin-Schreier cover is determined by the conductor of its unique branch point. We therefore reformulate Question \ref{questionlocal} as follows.

\begin{question}
\label{questionlocalgoodreduction}
Suppose $\{h_1, \ldots, h_r\} \in \Omega_{h}$. Does there exist a $\mathbb{Z}/p$-cover of $\spec R[[X]]$ with good reduction that has generic branching datum $[h_1, \ldots, h_r]^{\top}$?
\end{question}

\subsection{Refined Swan conductors}
\label{secrefinedswan}
Suppose $\psi$ is a cyclic cover of a closed disc $\spec R\{X\}$ with valuation $\nu_0$. After enlarging $R$, we may assume that $\psi$ is weakly unramified with respect to $\nu_0$. The set of all branch points is called the \textit{branch locus} of $\psi$ and is denoted by $\mathbb{B}(\psi)$. 

We define two invariants that measure the ramification of $\psi$ with respect to the valuation $\nu_0$. The \textit{depth} is
\[ \delta(\psi):=\sw(\psi)/p \in \mathbb{Q}_{\ge 0}, \]
\noindent where $\sw(\psi)$ is the Swan conductor defined identically to one for mixed characteristic \cite[Definition 3.3]{MR3167623} of the character associated to $\psi$. The rational number $\delta(\psi)$ is equal to $0$ if and only if $\psi$ is unramified. If this is the case, then its reduction $\psi$ is well-defined. In particular, if $\psi$ is of order $p$ and $\delta(\psi)=0$, then there exists $u \in \kappa \setminus \wp\kappa$ such that $\phi$ is defined by $y^p-y=u$. We call $u$ the \textit{reduction} of $\psi$. Recall that, by Artin-Schreier theory, $u$ is unique up to adding an element of the form $\wp(a)=a^p-a$, where $a \in \kappa$. We say $\psi$ is \textit{radical} if $\delta(\psi)>0$.

Suppose $\delta(\psi)>0$ and $\pi\in R$ with $\nu(\pi)=1$. Then we can define the \textit{differential Swan conductor} or \textit{differential conductor} with respect to $\pi$
\[ \omega(\psi):=\dsw(\psi) \in \Omega^1_{\kappa}  \]
    \noindent in the same way as \cite[Definition 3.9]{MR3167623} (for mixed characteristic). It is derived from the \textit{refined Swan conductor} $\rsw(\psi)$ and depends on the choice of $\pi$. In particular, we have the relation
    \begin{equation}
        \label{eqnrefinedswan}
        \rsw(\psi)= \pi^{-sw(\psi)} \otimes \dsw(\psi) \in \mathfrak{m}^{-\sw(\psi)} \otimes_{\mathcal{O}_{\mathbb{K}}} \Omega^1_{\kappa},
    \end{equation}
where $\mathfrak{m}=(\pi)$. Note that $\rsw(\psi)$ does not depend on the choice of $\pi$. In this paper, we always pick $\pi=t$.

We call $\delta(\psi\lvert_{\mathcal{D}[s,z]})$ (resp. $\omega(\psi\lvert_{\mathcal{D}[s,z]})$) the depth Swan conductor (resp. the differential Swan conductor) at the \textit{place} $[s,z]$, or just at the place $s$ when $z=0$. We call the ordered pair $(\delta(\psi), \omega(\psi))$ when $\delta(\psi)>0$, or $(\delta(\psi), u)$ when $\delta(\psi)=0$ and $u$ the reduction of $\psi$, the \textit{degeneration type} of $\psi$. We usually refer to $\delta(\psi)$ and $\omega(\psi)$ as the \textit{refined Swan conductors} of $\psi$.

Suppose $\overline{x} \in \overline{\mathcal{D}}$ or $\overline{x}=\overline{\infty}$, and let $\ord_{\overline{x}}: \kappa^{\times} \xrightarrow{} \mathbb{Z}$ be a normalized discrete valuation whose restriction to $k$ is trivial. Then the composite of $v$ with $\ord_{\overline{x}}$ is a valuation on $K$ of rank $2$, which we denote by $\mathbb{K}^{\times} \xrightarrow{} \mathbb{Q} \times \mathbb{Z}$. In \cite{MR904945}, Kato defines a Swan conductor $\sw^K_{\psi} (\overline{x}) \in \mathbb{Q}_{\ge 0} \times \mathbb{Z}$. Its first component is equal to $\sw(\psi)$.
We define the \textit{boundary Swan conductor} with respect to $\overline{x}$
\[ \sw_{\psi}(\overline{x}) \in \mathbb{Z} \]
\noindent as the second component of $\sw^K_{\psi}(\overline{x})$. Geometrically, it gives the instantaneous rate of change of the depth in the direction (mentioned in Example \ref{examplesemistablemodel}) corresponding to $\overline{x}$. See Remark \ref{remarkboundaryswan} and \S \ref{secgoodannulus} for further discussion.

\begin{remark}
\label{remarkboundaryswan}
The invariant $\sw_{\psi}(\overline{x})$ is determined by $\delta(\psi)$ and $\omega(\psi)$ as follows:
\begin{enumerate}
    \item \label{remarkboundaryswan1} If $\delta(\psi)=0$, then
    \[ \sw_{\psi}(\overline{x})=\sw_{\phi}(\overline{x}) \]
    where $\phi$ is the reduction of $\psi$ and $\sw_{\phi}(\overline{x})$ is the usual Swan conductor of $\phi$ with respect to the valuation $\ord_{\overline{x}}$ \cite[IV,\S 2]{MR554237}. That follows immediately from the definitions. We thus have $\sw_{\psi}(\overline{x})\ge 0$ and $\sw_{\psi}(\overline{x})=0$ if and only if $\phi$ is unramified with respect to $\ord_{\overline{x}}$. Note that, if $\phi$ is a one point $\mathbb{Z}/p$-cover of $\spec k[[x]]$ with ramification jump $d$ at $\overline{0}$, then $\sw_{\phi}(\overline{0})$ is equal to $d$ \cite[Cor. 2 to Th. 1]{MR554237}. Hence, a $\mathbb{Z}/p$-cover $\psi$ of $R[[X]]$ is a deformation of $\phi$ over $R$ only if $\sw_{\psi}(\overline{0})=d$.
    \item \label{remarkboundaryswan2} If $\delta(\psi)>0$, then we have
    \[ \sw_{\psi}(\overline{x})=-\ord_{\overline{x}}(\omega(\psi))-1 \]
    \noindent This follows from \cite[Corollary 4.6]{MR904945}.
\end{enumerate}
\end{remark}

\subsubsection{Vanishing cycle formula}

This section adapts \cite[\S 5.3.3]{MR3194815} and generalizes it to the equal characteristic case.

\begin{definition}
A $G$-cover $\psi: Y \xrightarrow{} C \cong \mathbb{P}^1_K= \Proj K[X]$ is called \textit{admissible} if its branch locus $\mathbb{B}(\psi)$ is contained in the open disc $D$ corresponding to $\spec R[[X]]$. 
\end{definition}

\noindent By \S \ref{secreducetolocal}, we can restrict our study to the case where $\psi$ is admissible.

An affinoid subdomain $\mathcal{D} \subseteq C^{\an}$ gives rise to a blowup $C'_R \rightarrow C_R$ with the following properties (\cite{MR1202394}): $C'_R$ is a semistable curve whose special fiber $\overline{C}':=C'_R \otimes_R k$ consists of two smooth irreducible components that meet in exactly one point. The first component is the strict transform of $\overline{C}$, which we may identify with $\overline{C}$. The second component is the exceptional divisor $\overline{Z}$ of the blow-up $C'_R \rightarrow C_R$. It is isomorphic to the projective line over $k$ and intersects $\overline{C}$ in the distinguished point $\overline{x}_0$. By construction, the complement $\overline{Z}^0:= \overline{Z} \setminus \{ \overline{x}_0 \}$ is identified with the canonical reduction $\overline{\mathcal{D}}$ of the affinoid $\mathcal{D}$. This means that the discrete valuation on $\mathbb{K}$ corresponding to the prime divisor $\overline{Z} \subset C'_R$ is equivalent to the Gauss valuation $\nu$ associated with $\mathcal{D}$ (see. \S \ref{secsemistablemodel}) and that its residue field $\kappa$ may be identified with the function field of $\overline{Z}$. 

Let $Y'_R$ denote the normalization of $C'_R$ in $Y$. We obtain the following commutative diagram
\[
\begin{tikzcd}
Y'_R \arrow{r} \arrow[swap]{d} & Y_R \arrow{d} \\
C'_R \arrow{r} & C_R
\end{tikzcd}
\]
\noindent in which the vertical maps are finite $G$-covers and each horizontal map is the composition of an a blowup with a normalization. Let $\overline{W} \subset Y'_R$ be the exceptional divisor of $Y'_R \rightarrow Y_R$. After enlarging the ground field $K$, we may assume that $\psi$ is weakly unramified with respect to $\nu$ (the valuation corresponding to $\mathcal{D}$). Note that this holds if and only if $\overline{W}$ is reduced. That condition allows us to define the refined Swan conductors for the restriction of $\psi$ to $\mathcal{D}$ as in \S \ref{secrefinedswan}. We now choose a closed point $\overline{x} \in \overline{Z}^0=\overline{\mathcal{D}}$ and a point $\overline{y} \in \overline{W}$ lying over $\overline{x}$. We let
\[ U(\mathcal{D},\overline{x}):= ] \overline{x} [_{\mathcal{D}} \]
\noindent denote the residue class of $\overline{x}$ on the affinoid $\mathcal{D}$. Clearly, $U(\mathcal{D}, \overline{x})$ is isomorphic to the open unit disc. Finally, we let $q: \widetilde{\overline{W}} \xrightarrow{} \overline{W}$ denote the normalization of $\overline{W}$ and set
\[ \delta_{\overline{y}}:=\dim_k (q_{\ast}\mathcal{O}_{\widetilde{\overline{W}}}/\mathcal{O}_{\overline{W}})_{\overline{y}}. \]
\noindent Then $\delta_{\overline{y}} \ge 0$ and $\delta_{\overline{y}}=0$ if and only if $\overline{y} \in \overline{W}$ is a smooth point.

The above notation extends to the case $\mathcal{D}=C$ as follows. If $\mathcal{D}=C$, then we let $\overline{Z}:= \overline{C}$ denote the special fiber of the smooth model $C_R$ of $C$ and $\overline{W}:= \overline{Y}$ the special fiber of $Y$. We set $\overline{x}:= \overline{x}_0$ and choose an arbitrary point $\overline{y} \in \overline{W}$ above $\overline{x}_0$. The residue class $U(\mathcal{D},\overline{x})$ is now equal to the open disc $D$, and the invariant $\delta_{\overline{y}}$ is defined in the same way as for $\mathcal{D} \subsetneq C$.

\begin{definition}
Suppose $\psi$ is a $\mathbb{Z}/p$ of $C$, $\mathcal{D} \subseteq C$ is an affinoid, and $\overline{x} \in \overline{\mathcal{D}}$. We define
\[ \mathfrak{C}(\mathcal{D},\psi,\overline{x}):=\bigg\{ \sum_z h_{z} \mid z \in \mathbb{B}(\psi) \cap U(\mathcal{D},\overline{x})   \bigg\}. \]
\noindent where $\mathbb{B}(\psi)$ is the set of branch points of $\psi$, and $h_z$ is the conductor (which is the ramification jump plus one) of the branch point $z$.

\end{definition}

\begin{proposition}
\label{propvanishingcycle}
With the notation as above, we have
\[ \sw_{\psi\lvert_\mathcal{D}}(\overline{x})=\mathfrak{C}(\mathcal{D},\psi,\overline{x})-1-\frac{2\delta_{\overline{y}}}{p-1}. \]
\end{proposition}

\begin{proof}
The proposition is the equal characteristic analog of \cite[Proposition 2.14]{MR3552986}, where everything but the characteristic of $R$ is identical. One thus can imitate the proof of \cite[Proposition 2.7]{MR3552986} (which gives a more general version of \cite[Proposition 2.14]{MR3552986}) when $n=1$. The only differences is $\lvert \mathbb{B} \cap U(\overline{x}) \rvert$ in Obus and Wewers' formula is the integer $\mathfrak{C}(\mathcal{D},\psi,\overline{x})$ in our situation, as they are both equal to $\sum_{y \in \text{Ram}(f)\cap V} R_y$ in \cite[Theorem 6.2.3]{MR3552539}, which translates to $\sum_{w \in Y \mid \psi(w) \in  \mathbb{B} \cap U(\overline{x})} R_w$ in ours. That number $R_w$ is, in turn, equal to the degree of the different of the corresponding local extension \cite[Theorem 4.6.4]{MR3552539}, hence is $(p-1)h_{\psi(w)}$ (recall that $h_{\psi(w)}$ is the conductor of the branch point $\psi(w)$).
\end{proof}


We then obtain the following results. 

\begin{corollary}
\label{corgood}
Let $\psi$ be an admissible cyclic cover of $C$ of order $p$. Then
   $$ \mathfrak{C}(\mathcal{D},\psi, \overline{0})=\sum_{z \in \mathbb{B}(\psi)} h_z \ge \sw_{\psi\lvert_{\mathcal{D}}}(\overline{0})+1. $$ 
Also, the cover $\psi$ has good reduction if and only if $\delta({\psi})=0$ and equality holds above.
    
    
    

\end{corollary}

\begin{proof}

The proof adapts the one for \cite[Corollary 5.12]{MR3194815}. The inequality in the first assertion follows immediately from Proposition \ref{propvanishingcycle} since $\mathbb{B}(\psi) \subsetneq D=U(\mathcal{D}[0],\overline{0})$ (because we assume $\psi$ is admissible). In addition, the cover $\psi$ has good reduction if and only if $\delta(\psi)=0$ and $\overline{Y}=\overline{W}$ is smooth in any point $\overline{y}$ above the distinguished point $\overline{0}$. The latter condition is equivalent to $\delta_{\overline{y}}=0$. Thus, the rest also follows from Proposition \ref{propvanishingcycle}. 
\end{proof}

\begin{remark}
In the situation of Corollary \ref{corgood}, when $\psi$ has good reduction, the reduction $\phi$ of its restriction to $D$ is a one-point-cover of $k[[t]]$ of conductor $\mathfrak{C}(\mathcal{D},\psi, \overline{0})$. That is because $\sw_{\psi\lvert_{\mathcal{D}}}(\overline{0})$ coincides with $\sw_{\phi}(\overline{0})$ (Remark \ref{remarkboundaryswan} (\ref{remarkboundaryswan1})), and the latter is equal to the ramification jump of $\phi$ at $\overline{0}$. Corollary \ref{corgood} is thus equivalent to the different criterion (Proposition \ref{propdifferentcriterion}). We will utilize the former for the rest of the paper. 
\end{remark}

\begin{corollary}
\label{corswgood}
Let $\psi$ be an admissible cyclic cover of $C$ of order $p$, let $\mathcal{D} \subseteq C$ be an affinoid, and let $\overline{x}$ be a point on the canonical reduction $\overline{\mathcal{D}}$ of $\mathcal{D}$. Then
\begin{equation*}
    \sw_{\psi\lvert_\mathcal{D}}(\overline{x}) \le \mathfrak{C}(\mathcal{D},\psi,\overline{x})-1.
\end{equation*}
\noindent Moreover, if $\psi$ has good reduction, then equality holds.
\end{corollary}

\begin{proof}
Parallel to the proof of \cite[Corollary 5.13]{MR3194815}.
\end{proof}

The following result follows immediately from Corollary \ref{corswgood} and Remark \ref{remarkboundaryswan}.

\begin{corollary}
\label{coromegagood}
In the same situation of Corollary \ref{corswgood}, if $\delta({\psi\lvert_{\mathcal{D}}})>0$ and $\overline{x} \neq \infty$, we have
\begin{equation*}
    \ord_{\overline{x}}(\omega(\psi\lvert_{\mathcal{D}})) \ge -\mathfrak{C}(\mathcal{D},\psi, \overline{x})
\end{equation*}
\noindent with equality if $\psi$ has good reduction. 

\end{corollary}

\begin{remark}
\label{remarkgooddiff}
Corollary \ref{corgood} and Corollary \ref{coromegagood} indicate that, for $\psi$ to have good reduction, $\omega({\psi\lvert_{\mathcal{D}}})$ cannot have any zeros besides one at infinity for any affinoid $\mathcal{D} \subseteq C$.
\end{remark}

\begin{remark}
\label{remarkgeneralizevanishingformula}
The results in this section easily generalize to all cyclic covers. We will discuss it in the "forthcoming paper" mentioned in Remark \ref{remarkdeformationASW}
\end{remark}

\subsubsection{Refined Swan conductor of Artin-Schreier extensions}

In this section, we study the depth and differential Swan conductors of an Artin-Schreier cover of a closed disc $\spec R\{X\}$. Suppose $\psi$ is a $\mathbb{Z}/p$-cover of the disc defined by
\[ Y^p-Y=u, \]
\noindent where $u \in \mathbb{K} \setminus \wp \mathbb{K}$.

\begin{definition}
Such an $u \in \mathbb{K}\setminus \mathcal{O}_{\mathbb{K}}$ is said to be \textit{reduced} if $u =\tau  \omega$, with $\tau \in \mathbb{K}$, $\nu(\tau)<0$, $\omega \in \mathcal{O}_\mathbb{K}^{\times}$, and $\overline{\omega} \not \in \kappa^p$. We say that $u$ is \textit{reducible} if there exists an $a \in \mathbb{K}^{\times}$ such that $u+a^p-a$ is reduced. 
\end{definition}

\begin{remark}
In our situation, as we can always replace $R$ by its finite extensions, we may assume that $u$ is reducible.
\end{remark}

\begin{proposition}
\label{proprefinedswanorderp}
Suppose $\psi$ is a $\mathbb{Z}/p$-cover of $R\{X\}$ defined by
\begin{equation}
\label{eqnrefinedswanorderp}
     Y^p-Y= u,
\end{equation}
\noindent where $u \in \mathbb{K}$ is reduced when $\nu(u)<0$. If $u=\pi^{-n} \omega$ ($\pi$ is a fixed uniformizer of $R$), with $\omega \in \mathcal{O}_\mathbb{K}^{\times}$, $\overline{\omega}\not\in \kappa^p$ , and $n \in \mathbb{Q}_{\ge 0}$, then the depth of $\psi$ is $n/p$ and 
\begin{enumerate}
    \item when $n =0$, the degeneration is $\overline{\omega}$,
    \item when $n >0$, its differential Swan conductor is $d\overline{\omega}$.
\end{enumerate}
\end{proposition}

\begin{proof}
The case $n=0$ is immediate from the definition of the refined Swan conductors. When $n>0$, we first describe a standard process to calculate $\rsw(\psi)$ (see (\ref{eqnrefinedswan})). 

\begin{proposition}
\label{propgeneralcalculaterefined}
Suppose $L/K$, which we denote by $\psi$, is a $\mathbb{Z}/p=\langle \sigma \rangle$-extension of complete discrete valuation field whose residue extension $\overline{L}/\overline{K}$ is purely inseparable. Fix an element $x \in \mathcal{O}_L^{\times}$ with $\overline{L}=\overline{K}[\overline{x}]$ and an element $\pi_K \in K$ with valuation one. Set $y:=N_{L/K}(x)$ (hence $\overline{y}=\overline{x}^p \in \overline{K}\setminus \overline{K}^p$), $a:=\sigma(x)/x-1$ and $b:=N_{L/K}(a)$. Then the refined Swan conductor of $\psi$ with respect to the valuation of $K$ is
\begin{equation}
    \label{eqnstandardrsw}
    \rsw(\psi)=b^{-1} \otimes \frac{d\overline{y}}{\overline{y}}\in \mathfrak{m}_K^{-\sw(\psi)} \otimes_{\mathcal{O}_K} \Omega^1_{\overline{K}}. 
\end{equation}
where $\mathfrak{m}_K=(\pi_K)$ is an ideal of $\mathcal{O}_K$.
\end{proposition}
\begin{proof}
By \cite[\S 3.3, Lemma 15]{MR603953}, with the notation as above, we have
\[ \{\psi, 1-bc\}_K= \{i_1(\overline{c}), y\}_K =i_2\bigg(\overline{c} \frac{d\overline{y}}{\overline{y}}\bigg) \]
for all $c \in \mathcal{O}_K$, where $\{\_ , \_ \}_K$, $i_1$, and $i_2$ are defined in \cite[\S 3]{MR904945}. The second equality follows from the definition of $i_2$. Equation (\ref{eqnstandardrsw}) then follows from the definition of $\rsw$ and \cite[Theorem 3.6 (2)]{MR904945}.
\end{proof} 
Let us now go back to the situation of Proposition \ref{proprefinedswanorderp}. One may rewrite (\ref{eqnrefinedswanorderp}) as
\[ Z^p-\pi^{n(p-1)/p}Z=\omega \]
simply by setting $Z=Y\pi^{n/p}$. Hence $\overline{Z}^p=\overline{\omega}$. The assumption $\overline{\omega} \not\in \kappa^p$ suggests that the residue field upstair is $\kappa[\overline{Z}]$. Set $x:=Z$. Then $a=\sigma(x)/x-1=\pi^{n/p}/Z=1/Y$, $y=N_{\mathbb{L}/\mathbb{K}}(x)=\omega$, and $b=N_{\mathbb{L}/\mathbb{K}}(a)=\pi^n/\omega$. Apply Proposition \ref{propgeneralcalculaterefined}, we obtain
\[ \rsw(\psi)= \frac{\omega}{\pi^n} \otimes \frac{d \overline{\omega}}{\overline{\omega}}= (\pi^n)^{-1} \otimes d\overline{\omega}. \]
The rest then follows from the definition of the conductors in \S \ref{secrefinedswan}.
\end{proof}

\begin{remark}
The proposition implies, if $\psi$ is an Artin-Schreier cover where $\delta(\psi)>0$, then its differential conductor $\omega(\psi)$ is an exact differential form. That is no longer true, however, for a cyclic $G$-cover (for instance, $G=\mathbb{Z}/p^2$) in general. More details will also be given in the "forthcoming paper" to which Remark \ref{remarkdeformationASW} refers.
\end{remark}

\begin{example}
\label{examplecomputeswan}

Recall that, in example \ref{examplemain}, the cover $\psi$ is defined by
\begin{equation}
\label{eqnr0}
   Y^5-Y=\frac{-2X+t^{10}}{(-2)X^5(X-t^{10})^2(X-t^5)^5}. 
\end{equation}
Fix a valuation $\nu_0$ such that $\nu_0(t)=1$. As the generic fiber is branched at three points $X=0, X=t^5$, and $X=t^{10}$, we follow \S \ref{secsemistablemodel} to construct the dual graph of the associated semistable model as follows.
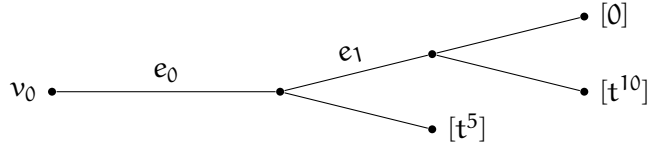
\tikzstyle{level 1}=[level distance=3cm, sibling distance=2cm]
\tikzstyle{level 2}=[level distance=2cm, sibling distance=1cm]
\tikzstyle{level 2}=[level distance=2cm, sibling distance=1cm]
\tikzstyle{bag} = [text width=6em, text centered]
\tikzstyle{end} = [circle, minimum width=3pt,fill, inner sep=0pt]
\[ 
\begin{tikzpicture}[grow=right, sloped]
\node[end, label=left:{$v_0$}]{}
child{
        node[end]{}
    child {
                node[end, label=right:
                    {$[t^5]$}] {}
                edge from parent
            }
    child {
        node[end]{}        
            child {
                node[end, label=right:
                    {$ [t^{10}]$}] {}
                edge from parent
            }
            child {
                node[end, label=right:
                    {$ [0]$}] {}
                edge from parent
            }
            edge from parent 
            node[above] {$e_1$}
    }
    edge from parent
    node[above] {$e_0$}
    };
   
\end{tikzpicture}
\]
\noindent We set the orientation of the graph to be the direction from its leaves to its root $v_0$.

At the initial vertex of $e_1$, which corresponds to a punctured disc of radius $5^{-2}$ defined by $\spec R\{t^{-10}X,t^{10}X^{-1},t^{10}(X-t^{10})^{-1} \}$, we set $X_2:=Xt^{-10}$. Replacing $X$ by $X_2$ equates to looking at the restriction of the cover to the sub-disc $\mathcal{D}[10]= \spec R\{t^{-10}X\}$. We rewrite (\ref{eqnr0}) as
 \begin{equation}
 \label{eqnr2}
     Y^5-Y=\frac{-2X_2+1}{(-2)X_2^5(X_2-1)^2(X_2t^5-1)^5t^{5\cdot 17 }}
 \end{equation}
\noindent Let $x_2$ the image of $X_2$ in the residue field of $\mathcal{D}[10]$. The derivative of the reduction modulo $5$ of $\frac{-2X_2+1}{(-2)X_2^5(X_2-1)^2(X_2t^5-1)^5}$ is $\omega(2)=\frac{dx_2}{x_2^4(x_2-1)^3}$. Hence, it follows from Proposition \ref{proprefinedswanorderp} that $\omega(2)$, after replacing $x_2$ by $x$, is the differential Swan conductor of $\psi$ restricting to $\mathcal{D}[10]$, and its depth conductor is $17$. Its boundary Swan conductors are $\sw_{\psi \lvert_{\mathcal{D}[10]}}(\overline{\infty})=-\ord_{\infty}(\omega(2))-1=-6$, $\sw_{\psi\lvert_{\mathcal{D}[10]}}(x_2)=-\ord_{x_2}(\omega(2))-1=3$, and $\sw_{\psi \lvert_{\mathcal{D}[10]}}(x_2-1)=-\ord_{x_2-1}(\omega(2))-1=2$.

At the initial vertex of $e_0$, which resembles the punctured $\spec R\{t^{-5}X, t^{5}X^{-1}, t^5(X-t^5)^{-1} \}$. Define $X_1:=Xt^{-5}$. We then rewrite (\ref{eqnr0}) as
 \[Y^5-Y=\frac{-2X_1+t^5}{(-2)X_1^5(X_1-t^5)^2(X_1-1)^5t^{5\cdot 11}}.\]
\noindent Apply the same argument as above, we see that the degeneration type of $\psi$ on $\spec R\{X_1\}$ is $\Big(11, \frac{dx}{x^7(x-1)^5}\Big)$.

Finally, when $r=0$, then $X_0=X$, the depth is $0$, and the reduction $\phi$ is defined by $y^5-y=\frac{1}{x^{11}}$. Therefore, we have $\mathfrak{C}(\psi, \overline{0})=5+3+4=12=\sw_{\phi}(\overline{0})+1$. Hence, $\psi$ has good reduction by Corollary \ref{corgood}.
\end{example}

In the next example, we will go into detail of describing how the degeneration data vary along the dual graph.

\begin{example}
\label{exnonflat}

Consider the $\mathbb{Z}/2$-cover $\psi$ of $\spec R\{X\}$ that is given by 
\begin{equation}
\label{eqn1june12}
    Y^2-Y =\frac{1}{X(X-t^2)}:=f(X,t). 
\end{equation}

\noindent The associated dual graph of the semistable model is as follows.
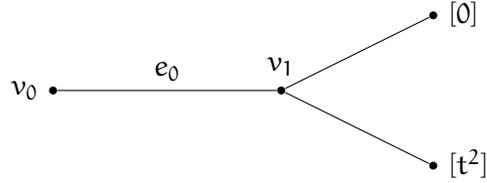
\tikzstyle{level 1}=[level distance=3cm, sibling distance=3cm]
\tikzstyle{level 2}=[level distance=2cm, sibling distance=2cm]
\tikzstyle{bag} = [text width=6em, text centered]
\tikzstyle{end} = [circle, minimum width=3pt,fill, inner sep=0pt]
\[ 
\begin{tikzpicture}[grow=right, sloped]
\node[end, label=left:{$v_0$}]{}
child{
        node[end, label=above:{$v_1$}]{}
    child {
                node[end, label=right:
                    {$[t^2]$}] {}
                edge from parent
            }
    child {
                node[end, label=right:
                    {$[0]$}] {}
                edge from parent
            }
    edge from parent
    node[above] {$e_0$}
    };
\end{tikzpicture}
\]
\noindent Recall from \S \ref{secsemistablemodel} that the vertex $v_1$ represents the subdisc $\mathcal{D}[2]:=\spec R\{t^{-2}X\}$. Substitute $t^{-2}X$ by $X_1$ in (\ref{eqn1june12}), we obtain
\begin{equation*}
    f(X_1,t)=t^{-2 \cdot 2} \frac{1}{X_1(X_1-1)}.
\end{equation*}
\noindent As $1/(x_1(x_1-1)) \notin (k(x_1)^{\times})^2$, we obtain $\delta(\psi\lvert_{\mathcal{D}[2]})=2$ and $\omega(\psi\lvert_{\mathcal{D}[2]})=\frac{dx}{x^2(x-1)^2}$.

Set $X_{r}:=Xt^{-2r}$. If $r > 1$, we can write $f(X,t)$ as
\[ f(X_{r},t)=t^{-2(r+1)}\frac{1}{X_{r}(X_{r}t^{2r-2}-1)}. \]
\noindent Again, as $ \frac{-1}{x_{r}} \notin (k(x_{r})^{\times})^2$, Proposition \ref{proprefinedswanorderp} shows $\delta(\psi\lvert_{\mathcal{D}[2r]})=r+1$ and $\omega(\psi\lvert_{\mathcal{D}[2r]})=\frac{dx}{x^2}$.

Suppose $r<1$. Then $[f]_{2r}$ is always a square. Set $a:=1/X^2$. We get
\[ f(X,t)+a^2-a= \frac{X^2+t^2X+t^2}{X^2(X-t^2)}=:g(X,t).\]
\noindent Moreover, we can rewrite $g(X,t)$ as
\[ g(X_{r},t)= t^{-4r} \frac{t^{2-2r}-t^{2}X_{r}+X_{r}^2t^{2r}}{X_{r}^2(X_{r}-t^{2-2r})}. \]
\noindent Note that $2-2r<2r$ for $r \in (1/2,1]$ and $2-2r>2r$ otherwise. Hence, one can describe $\delta(\psi\lvert_{\mathcal{D}[2r]})$ in terms of $r$ as follows,
\begin{equation*}
    \delta(\psi\lvert_{\mathcal{D}[2r]})= \begin{cases} 
      r & 0 \le r \le \frac{1}{2}, \\
      3r-1 & \frac{1}{2} \le r \le 1, \\
      r+1 & 1 \le r. 
   \end{cases}
\end{equation*}
\noindent Below is the graph of $\delta(\psi\lvert_{\mathcal{D}[2r]})$ with respect to $r$. Note that it is a piecewise linear function with respect to $r$. This phenomenon will be discussed further in \S \ref{secgoodannulus}.
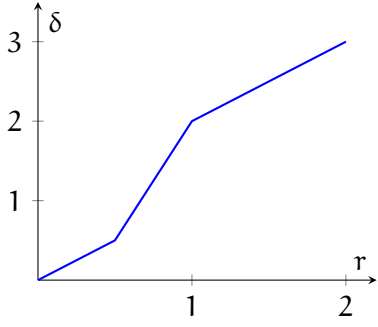
\begin{figure}[ht]
    \centering
\begin{tikzpicture}[
  declare function={
    func(\x)= and(\x >= 0, \x < 1/2) * (x)   +
              and(\x >= 1/2, \x < 1) * (3*\x-1)     +
              and(\x >= 1, \x < 3) * (\x+1)
   ;
  }
]
\begin{axis}[
  axis x line=middle, axis y line=middle,
  ymin=0, ymax=3.5, ytick={0,...,3}, ylabel=$\delta$,
  xmin=0, xmax=2.2, xtick={0,...,2}, xlabel=$r$,
  domain=0:2,samples=101, scale=0.65 
]

\addplot [blue,thick] {func(x)};
\end{axis}
\end{tikzpicture} 
    \caption{Graph of $\delta(\psi\lvert_{\mathcal{D}[2r]})$}
    \label{figgraphexnoflat}
\end{figure}

\noindent Moreover, we can write down the equation of $\omega(\psi\lvert_{\mathcal{D}[2r]})$ with respect to $r$ as follows.
\begin{equation*}
    \omega(\psi\lvert_{\mathcal{D}[2r]})= \begin{cases} 
      \frac{dx}{x^2} & 0 < r < \frac{1}{2}, \\
      \frac{(1+x^2)dx}{x^4} & r=\frac{1}{2}, \\
      \frac{dx}{x^4} & \frac{1}{2} < r < 1, \\
      \frac{dx}{x^2(x-1)^2} & r=1,\\
      \frac{dx}{x^2} & 1 < r. 
   \end{cases}
\end{equation*}
\noindent When $r=0$, the reduction $\phi$ of $\psi$ is an Artin-Schreier cover given by $y^2-y=\frac{1}{x^2}$. As $1/x^2$ is a square, we may replace the right-hand-side of the previous equation by $1/x^2+(1/x)^2-(1/x)=1/x$. Thus, by Corollary \ref{corgood}, $\psi$ does not have good reduction as $\sum_{x \in \mathbb{B}(\psi)} h_x =2+2 > \sw_{\psi}(\overline{0})+1=\sw_{\phi}(\overline{0})+1=1+1$. 
\end{example}

\begin{remark}
\label{remarksummarizeexnonflat}
For future reference, we list down some important degeneration data of the cover $\psi$ in Example \ref{exnonflat} here. The degeneration type of $\psi$ (resp. $\psi\lvert_{\mathcal{D}[2]}$) is $\Big(0, \frac{1}{x} \Big)$ (resp. $\Big(2, \frac{dx}{x^2(x-1)^2} \Big)$).
\end{remark}

\begin{remark}
If $f(X,t)$ in (\ref{eqn1june12}) is replaced by $\frac{1}{X^2(X-t^2)}$, then the deformation is flat of type $[4] \xrightarrow{}[2,2]^{\top}$. This type of deformation will be discussed further in \S \ref{secOSS}. 
\end{remark}

In the next two sections, we will study the degeneration of $\mathbb{Z}/p$-covers with good reduction on sub-discs and sub-annuli of the formal disc.

\subsubsection{Degeneration data on the boundary}

The following is a result by Sa{\"i}di, which characterizes $\mathbb{Z}/p$-covers of a boundary of a disc.

\begin{proposition}[c.f. {\cite[Proposition 2.3.1]{MR2377173}}]
\label{propdiscbound}
Let $A:=R[[X^{-1}]]\{X\}$, and let $\psi:\spec B \rightarrow \spec A$ be a nontrivial Galois cover of degree $p$. Assume that the ramification index of the corresponding extension of discrete valuation rings equals $1$. Then $\psi$ is a torsor under a finite flat $R$-group scheme $G_R$ of rank $p$ and the following cases occur:

\begin{enumerate}
    \item \label{etaleboundary} Suppose the depth is $\delta=0$. Then $\psi$ is a torsor under the {\'e}tale group $(\mathbb{Z}/p \mathbb{Z})_R$. Moreover, for a suitable choice of the parameter $X$ of $A$, the torsor $f$ is given by an equation $Y^p-Y=X^m$ for $m \in \mathbb{Z}$, prime to $p$. In addition, $Y^{1/m}$ is a parameter of $B$.
    \item \label{additiveboundary} Suppose the depth is $\delta>0$. Then $\psi$ is a torsor under the group scheme $\mathcal{M}_{\delta p,R}$. Moreover, for a suitable choice of the parameter $X$, the torsor $\psi$ is given by an equation $Y^p-Y=X^m/\pi^{p \delta}$ for $m \in \mathbb{Z}$, prime to $p$. In addition, $(Y\pi^{\delta})^{1/m}$ is a parameter of $B$.
\end{enumerate}
\end{proposition}

The proposition thus shows, in the language of this paper, that an order $p$ cover of $R[[X^{-1}]]\{X\}$ is determined by its depth $\delta$ and its boundary conductor $\sw_{\psi\lvert_{\mathcal{D}}}(\overline{0})=m$ (where $\mathcal{D}$ is the disc associated with $\spec R\{X\}$). We say it has \textit{degeneration type} $(\delta,-m)$.

\begin{remark}
\label{remarkboundaryswangoodreduction}
As $\spec A$ can be thought of as the boundary of the unit $\spec R[[X]]$, we will see in the next section that the number $m$ may also be considered as the instantaneous rate of change of the depth on the centrifugal direction. In addition, it follows from Corollary \ref{coromegagood} that, for $\psi$ to be a restriction to a subdisc of a cover (of a disc) with good reduction, it is necessary that $m=-\mathfrak{C}(\mathcal{D},\psi,\overline{0})+1<0$.
\end{remark}

\begin{remark}
Sa{\"i}di's result is motivated by \cite[Corollaire 1.8]{2000math.....11098H} for the case $R$ is of mixed characteristic. However, it is not true that a $\mathbb{Z}/p^n$-cover (where $n>1$) of a boundary is determined by its depth and boundary conductor in both mixed and equal characteristic cases. See \cite[\S 5.3]{brewisthesis} for a counterexample to the mixed characteristic case. We will discuss this phenomenon further in a forthcoming paper.
\end{remark}

\subsubsection{Good degeneration data on an annulus}
\label{secgoodannulus}
Fix a closed disc $\mathcal{D}$. For a $K$-point $z$ on $\mathcal{D}$ and $r \in \mathbb{Q}_{\ge 0}$, we denote by $\delta_{\psi}(r,z)$ (resp. $\delta_{\psi}(r)$) and $\omega_{\psi}(r,z)$ (resp. $\omega_{\psi}(r)$) the depth and the differential conductors of the restriction of $\psi$ to $\mathcal{D}[r,z]$ (resp. $\mathcal{D}[r]$). When the point $z$ is fixed, we may regard $\delta_{\psi}(r,z)$ and $\omega_{\psi}(r,z)$ as  functions in terms of $r$. Suppose $\overline{y}$ is a point on the canonical the reduction of $\mathcal{D}[r,z]$, or a point a infinity $\overline{y}=\overline{\infty}$. Set $\sw_{\psi}(r,z,\overline{y}):=\sw^K_{\psi\lvert_{\mathcal{D}[r,z]}}(\overline{y})$. The following results show how understanding the differential conductor can give a lot of information about the depth. They are the exact analogs of ones from \cite[5.3.2]{MR3194815} in mixed characteristic.

\begin{proposition}
\label{propdeltalinear}
Suppose $z \in K$ is fixed. Then $\delta_{\psi}(\_, z)$ extends to a continuous, piecewise linear function 
\[ \delta_{\psi}(\_, z): \mathbb{R}_{\ge 0} \xrightarrow{} \mathbb{R}_{\ge 0}. \]
\noindent Furthermore,
\begin{enumerate}
    \item \label{propdeltalinear1} For $r \in \mathbb{Q}_{>0}$, the left (resp. right) derivative of $\delta_{\psi}(\_, z)$ at $r$ is $-\sw_{\psi}(r,z,\overline{\infty})$ (resp. $\sw_{\psi}(r,z,\overline{0})$).
    \item \label{propdeltalinear2} If $r$ is a kink of $\delta_{\psi}(\_, z)$ (meaning the left and right derivaties do not agree), then $r \in \mathbb{Q}$.
\end{enumerate}
\end{proposition}

\begin{proof}
See \cite[Proposition 2.9]{2005math.....11434W}.
\end{proof}

\begin{definition}
\label{defncondopendisc}
Thanks to the above fact, it makes sense to define the refined Swan conductors of the restriction of $\psi$ to an open disc $D[r,z]$ as $\delta(\psi\lvert_{D[r,z]}):=\delta(\psi\lvert_{\mathcal{D}[r,z]})$ and $\omega(\psi\lvert_{D[r,z]}):=\omega(\psi\lvert_{\mathcal{D}[r,z]})$. We set the boundary conductor of $\psi\lvert_{D[r,z]}$ to be that of $\psi\lvert_{\mathcal{D}[r,z]}$.
\end{definition}

\begin{corollary}
\label{corleftrightderivative}
If $r \ge 0$ and $\delta_{\psi}(r,z)>0$, then the left and right derivatives of $\delta_{\psi}$ at $r$ are given by $\ord_{\overline{\infty}}(\omega_{\psi}(r))+1$ and $-\ord_{\overline{0}}(\omega_{\psi}(r))-1$, respectively. In particular, the integer $-\ord_{\overline{x}}(\omega_{\psi}(r,z))-1$ is the instantaneous rate of change of $\delta_{\psi}(r,z)$ on the direction with respect to $\overline{x}$ (defined in Example \ref{examplesemistablemodel}).
\end{corollary}

\begin{proof}
Immediate from Proposition \ref{propdeltalinear} (\ref{propdeltalinear1}) and Remark \ref{remarkboundaryswan} (\ref{remarkboundaryswan2}).
\end{proof}

\begin{remark}
It follows from the above results and Corollary \ref{corswgood} that, if a cover of a disc has good reduction, then its depth varies from a point inside the disc to its boundary like a piecewise linear, weakly concave down function. Moreover, the rate of change of the depth at each place is determined by the sum of conductors of the branch points further away from the boundary (i.e. ones inside the corresponding closed disc). It is not the case in Example \ref{exnonflat} because $\delta(\psi\lvert_{\mathcal{D}[r]})$, regarded as a function with respect to $r$, fails to concave down at $r=1/2$ as showed in Figure \ref{figgraphexnoflat}.
\end{remark}

We thus can describe the degeneration of a cover of an annulus as below.

\begin{proposition}[{c.f. \cite[Proposition 3.3.9]{MR2377173}}]
\label{propgoodannulus}
Suppose $\psi$ is a $\mathbb{Z}/p$-cover of a unit disc $D$ with good reduction. Let $\mathcal{X}$ be an annulus of thickness $p\epsilon$ lying inside $D$ that contains no branch points of $\psi$. Let $S_1$ and $S_2$ be the inside (one further from the boundary of $D$) and the outside boundary of $\mathcal{X}$, respectively. Then the reduction on $S_1$ is of type $(d,q)$ (where $d \ge 0$) if and only if the degeneration type on $S_2$ is $(d-q\epsilon, -q)$.
\end{proposition}

\begin{proof}
Let $S, U \in \Frac R[[X]]$ be such that $\mathcal{X} \cong \spec R[[S,U]]/(SU-\pi^{p\epsilon})$. One may think of $\mathcal{X}$ as the complement of $\spec R\{U^{-1}\}$ inside $\spec R[[S]]$. Hence $S_2 \cong \spec R\{S\}[[S^{-1}]]$ and $S_1 \cong \spec R\{U^{-1}\}[[U]]$. As in \S \ref{secsemistablemodel}, we may associate with $\spec R[[S]]$ a unit disc $D$ over $K$ parameterized by $S$, and with $\spec R[[U^{-1}]]$ its sub-disc $D[p\epsilon]$. Set $\mathfrak{C}:=\mathfrak{C}(\mathcal{D}, \psi, \overline{0} )$. For $r \in \mathbb{Q} \cap [0,p\epsilon]$, define $\delta(r):=\delta({\psi\lvert_{\mathcal{D}[r]}})$, $\omega(r):=\omega({\psi\lvert_{\mathcal{D}[r]}})$. Apply Corollary \ref{coromegagood} and Corollary \ref{corleftrightderivative}, we obtain
\begin{equation}
    \label{eqndiffannulus}
    \ord_{\overline{0}}\omega(r)=-\mathfrak{C} \text{ for } r \in \mathbb{Q}\cap[0,p\epsilon).
\end{equation}
Thus, it must also be true that $\ord_{\overline{\infty}}\omega(r)=\mathfrak{C}-2 \text{ for } r \in \mathbb{Q}\cap(0,p\epsilon]$. Therefore, the boundary Swan conductors of $S_2$ and $S_1$ are $q:=-\mathfrak{C}+1$  and $-q$, respectively. Finally, it follows from (\ref{eqndiffannulus}) and Proposition \ref{propdeltalinear} that $\delta(\_)$ is linear with slope $\mathfrak{C}-1$ on $[0,p\epsilon]$. That requires the depth at the two boundaries to behave as in the statement. 
\end{proof}

\begin{remark}
\cite[Proposition 3.3.9]{MR2377173} is actually much stronger than Proposition \ref{propgoodannulus}. It claims that the two boundaries' degeneration types completely determine such a cover of an annulus. However, the author does not give the proof but only cites \cite[Proposition 4.2.5]{MR2032453}, where the corresponding result in mixed characteristic is proved. We also do not fully understand that proof at the moment as the approach is quite different from the one above. Ours is still good enough for the purpose of this paper, though.
\end{remark}

\section{Hurwitz trees of Artin-Schreier covers}
\label{sectionHurwitz}

Suppose $R$ is a complete discrete valuation ring whose residue field is of characteristic $p>0$. When $R$ is of mixed characteristic (e.g., $R=W(k)$), Henrio, Bouw, Wewers, and Brewis define a combinatorial object called \textit{Hurwitz tree} from a $G$-cover of a formal disc over $R$ (see \cite{2000math.....11098H}, \cite{MR2254623}, \cite{MR2534115}). This construction gives an obstruction for the lifting of Galois covers in characteristic $p$. In this section, we introduce the notion of Hurwitz tree in equal characteristic, which basically formalizes what Maugeais and Sa{\"i}di did in \cite{MR2015076} and \cite{MR2377173}, using the language from \cite{2000math.....11098H} and \cite{MR2254623}. Then, we will describe how a $\mathbb{Z}/p$-deformation gives rise to such a tree. Finally, we will provide an obstruction for the equal characteristic deformation of Artin-Schreier covers and compute the Hurwitz trees of the previous examples. 

Throughout the rest of the paper, we assume $R=k[[t]]$, and $K:=k((t))$ is the fraction field of $R$. 

\subsection{Hurwitz tree}
\label{sectconstruct}
The below definition and the two following paragraphs are almost identical to ones in \cite[\S 3.1]{MR2254623}. We recite here for the convenience of the reader.
\begin{definition}
\label{defdecorated}
A \emph{decorated tree} is given by the following data:

\begin{itemize}
    \item a semistable curve $C$ over $k$ of genus $0$,
    \item a family $(x_b)_{b\in B}$ of pairwise distinct smooth $k$-rational points of $C$, indexed by a finite nonempty set $B$,
    \item a distinguished smooth $k$-rational point $x_0 \in C$, distinct from any of the point $x_b$.
\end{itemize}
We require that $C$ is stably marked by the points $((x_b)_{b \in B},x_0)$.
\end{definition}

The \emph{combinatorial tree} underlying a decorated tree $C$ is the graph $T=(V,E)$, defined as follows. The vertex set $V$ of $T$ is the set of irreducible components  of $C$, together with a distinguished element $v_0$. The edge set $E$ is the set of singular points of $C$, together with a distinguished elements $e_0$. We write $C_v$ for the component corresponding to a vertex $v \neq v_0$ and $x_e$ for the singular point corresponding to an edge $e \neq e_0$. An edge $e$ corresponding to a singular point $x_e$ is adjacent to the vertices corresponding to the two components which intersect at $x_e$. The edge $e_0$ is adjacent to the root $v_0$ and the vertex $v$ corresponding to the (unique) component $C_v$ containing the distinguished point $x_0$. For each edge $e \in E$, the \emph{source} (resp. the \emph{target}) of $e$ is the unique vertex  $s(e) \in V$ (resp. $t(e) \in V$)  adjacent to $e$ which lies in the direction of the root (resp. in the direction away from the root).

Note that, since $(C, (x_b), x_0)$ is stably marked of genus $0$, the components $C_v$ have genus zero, too, and the graph $T$ is a tree. Moreover, we have $\lvert B \rvert \ge 1$. For a vertex $v \in V$, we write $U_v \subsetneq C_v$ for the complement in $C_v$ of the set of singular and marked points.

\begin{definition}
\label{defhurwitztree}
Let $\overrightarrow{E}=\{h_1,  \ldots, h_r\}$ an element of $\Omega_{\sum_{i=1}^r h_i}$ (defined in \S \ref{secmoduli}). A \textit{$\mathbb{Z}/p$-Hurwitz tree} (or just \textit{Hurwitz tree} in the context this paper) of \textit{type} $\overrightarrow{E}$ is given by the following data:
\begin{itemize}
    \item A decorated tree $C=(C,(x_b),x_0)$ with underlying combinatorial tree $T=(V,E)$.
    \item For every $v \in V$, a rational number $\delta_v \ge 0$, called the \emph{depth} of $v$.
    \item For each $v \in V \setminus \{v_0\}$, an exact differential form $\omega_v$ called the \textit{differential conductor} at $v$.
    \item For every $e \in E$, a positive rational number $\epsilon_e$, called the \emph{thickness} of $e.$
    \item For every $b_i \in B=\{b_1, \ldots, b_r\}$, the positive number $h_i$ (from $\overrightarrow{E}$), called the \textit{conductor} at $b$.
    \item For $v_0$, a fraction $\frac{1}{x^d}$, where $d \not\equiv 0 \pmod{p}$, called the \textit{degeneration} at $v_0$.
\end{itemize}

\noindent These objects are required to satisfy the following conditions.

\begin{enumerate}[label=(\text{H}{{\arabic*}})]
    \item \label{c1Hurwitz} Let $v \in V$. We have $\delta_v \neq 0$ if $v \neq v_0$.
    \item \label{c2Hurwitz} For each $v \in V\setminus \{v_0\}$, the differential form $\omega_v$ does not have zeros nor poles on $U_v \subsetneq C_v$.
    \item \label{c3Hurwitz} For every edge $e \in E - \{e_0\}$, we have the equality 
    \[ -\ord_{x_e} \omega_{t(e)}-1 = \ord_{x_e}\omega_{s(e)}+1 \not\equiv 0 \pmod{p} . \]
    \item \label{c4Hurwitz} For $v_0$, we have $d = \ord_{x_{e_0}}\omega_{t(e_0)}+1$.
    \item \label{c5Hurwitz} For every edge $e \in E$, we have
    \[\delta_{s(e)}+\epsilon_e  d_{e} = \delta_{t(e)}, \]
    \noindent where
    \[ d_{e}:= -\ord_{x_e} \omega_{t(e)}-1 \underset{s(e) \neq v_0}{\stackrel{\ref{c2Hurwitz}}{=}}\ord_{x_e} \omega_{s(e)}+1 .\]
    \item \label{c6Hurwitz} For $b \in B$, let $C_v$ be the component containing the point $x_b$. Then the differential $\omega_v$ has a pole in $x_b$ of order $h_b$. 
\end{enumerate}

\noindent For each $v \in V \setminus \{v_0\}$, we call $(\delta_v, \omega_v)$ the \textit{degeneration type} of $v$. For each $e \in E$, we call $(\delta_{s(e)}, -d_{s(e)})$ (resp. $(\delta_{t(e)}, d_{t(e)})$) the \textit{initial degeneration type} (resp. the \textit{final degeneration type}) of $e$. The integer $h:=d+1$ is called the \emph{conductor} of the Hurwitz tree. The rational $\delta:=\delta_{v_0}$ is the \emph{depth}. We define the \textit{height} of the tree to be the maximal number of edges on a direction from its root to its leaves.
\end{definition}

One can easily obtain the following result. See \cite[\S 2]{2000math.....11098H} and \cite{MR1645000} for the proofs of its mixed characteristic analogs.

\begin{lemma}
\label{lemmarootconductor}
Let $(C,\omega_v, \delta_v, \epsilon_e,  h_b, d)$ be a Hurwitz tree. Fix an edge $e \in E$ and let $C_e \subseteq C$ be the union of all components $C_v$ corresponding to vertices $v$ which are separated from the root $v_0$ by the edge $e$. Then:
    \[ d_e= \sum_{\substack{b \in B \\ x_b \in C_e}} h_b -1 >0.\]
\noindent In particular, $h=\sum_{b \in B} h_b$.
\end{lemma}


\subsection{The Hurwitz tree associated to a \texorpdfstring{$\mathbb{Z}/p$}{Zp}-cover of a disc.}
\label{seccovertotree}

Let $D_K:=\{X \mid   \lvert X \rvert_K<1 \}$ be the rigid open unit disc over $K$. Suppose we are given a $\mathbb{Z}/p$-cover $\psi$ of $D_K$, which is  algebraically given by
\[ \psi: R[[Z]] \xrightarrow{} R[[X]]\]
\noindent with good reduction. Hence $\delta(\psi)=0$. Obviously, $\psi$ induces a $\mathbb{Z}/p$-cover of the boundary $\spec \allowbreak R[[X^{-1}]]\{X\}$ of the disc $D_K$. Let $d$ be the boundary conductor of this cover. Following \cite{MR2534115}, we will now associate to $\psi$ a Hurwitz tree.

Let $x_{b,k}\in D_K$ be the branch points of $\psi$, indexed by the finite set $B$. We assume that the points $x_{b,K}$ are all $K$-rational and the conductor at $x_{b,K}$ is $h_b$. By the different criterion (Proposition \ref{propdifferentcriterion}) and Proposition \ref{propdiscbound}, we obtain $\sum_{b \in B} h_{b} = d+1=: h$. We assume that $\psi$ has at least two branch points $(\lvert B \rvert \ge 2)$, i.e. the deformation is non-trivial. Applying \S \ref{secsemistablemodel}, one obtains  $D_R$, the semistable model of the marked disc $(D_K, (x_{b,K})_{b \in B})$ with special fiber $(\overline{D},(x_b),x_0)$. Note that it is a decorated tree, in the sense of Definition \ref{defdecorated}.


Let $(V,E)$ be the combinatorial tree underlying $T:=(\overline{D},(x_b),x_0)$. For $v \in V$, let $U_v \subsetneq D_v$ be the complement of the singular and marked points and let $U_{v,K} \subsetneq D_K$ be the affinoid subdomain with reduction $U_v$. Let $V_v$ (resp. $V_{v,K}$) denote the inverse image of $U_v$ (resp. $U_{v,K}$) under $\psi$. By construction, $V_{v,K} \xrightarrow{} U_{v,K}$ is a torsor under the constant $K$-group scheme $G$. We set $\delta_v \in \mathbb{Q}_{\ge 0}$ to be the depth conductor of the restriction of the cover to $V_{v,K}$. When $\psi\lvert_{V_{v,K}}$ is radical, we set $\omega_v \in \Omega^1_{k(x)}(U_v)$ to be the differential conductor of the same restriction. When $\psi\lvert_{V_{v,K}}$ is {\'e}tale, which only happens when $v=v_0$, its reduction can be represented by a fraction $\frac{1}{x^d}$, where $d$ is prime to $p$, by Proposition \ref{localAS}. We set the degeneration at $v_0$ to be that fraction. Finally, for $e \in E$ we let $A_e \subseteq D_K$ denote the subset of all points which specialize to the singular point $x_e \in \overline{D}$ corresponding to $e$. This is an open annulus. We define $\epsilon_e$ as the thickness of $A_e$ divided by $p$, i.e., the positive rational number such that 
\[ A_e \cong \{ x \mid \lvert p \rvert^{p \epsilon_e}_K < \lvert x \rvert_K <1 \}. \]

\begin{proposition}
\label{propactiontotree}
The datum $(T, \omega_v, \delta_v, \epsilon_e, h_b, h-1)$ defines a Hurwitz tree of type $\{h_1, \ldots, h_r\}$. Moreover, $h$ is the conductor, $0$ is the depth.
\end{proposition}

\begin{proof}
The proof is parallel to one for the mixed characteristic case \cite[\S 3.1]{2000math.....11098H}. The only differences are: all differential conductors are exact, each branch point contributes more to the degree of the different (see Remark \ref{remarksumofconductors}), and the depth can be arbitrarily large.

\ref{c3Hurwitz} It follows from Proposition \ref{propgoodannulus}, as $-\ord_{x_e}\omega_{t(e)}-1$ is the inside boundary conductor of the associated annulus and $\ord_{x_e}\omega_{t(e)}+1$ is its outside boundary one. 

\ref{c5Hurwitz}, \ref{c4Hurwitz} Suppose $e$ is an edge with initial depth $\delta_{s(e)}$, final depth $\delta_{t(e)}$, and thickness $\epsilon_e$. Then, as the restriction of the cover to the corresponding annulus has good reduction, Proposition \ref{propgoodannulus} shows that $\delta_{t(e)}=\delta_{s(e)}+\epsilon_e d_e$. 

\ref{c2Hurwitz}, \ref{c6Hurwitz} follow from Remark \ref{remarkgooddiff}.

\ref{c1Hurwitz} It follows immediately from the definition of the depth that $\delta_{v_0} \ge 0$. Moreover, by \ref{c5Hurwitz}, $\delta_v$ is a strictly increasing function, as $v$ goes away from the root.

Finally, it is immediate from the construction and the good reduction assumption that the depth is $0$ and the conductor is $h$, completing the proof.
\end{proof}

\begin{example}
\label{excalculatehurwitzmain}
Let $\psi$ be the $\mathbb{Z}/5$-cover in Example \ref{examplemain}. Then it follows from Example  \ref{examplecomputeswan} that the Hurwitz tree associated with $\psi$ has the following form.

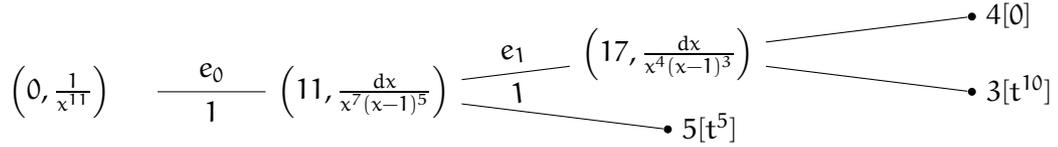
\tikzstyle{level 1}=[level distance=4cm, sibling distance=2cm]
\tikzstyle{level 2}=[level distance=4cm, sibling distance=1cm]

\tikzstyle{bag} = [text width=6em, text centered]
\tikzstyle{end} = [circle, minimum width=3pt,fill, inner sep=0pt]
\[ 
\begin{tikzpicture}[grow=right, sloped]
\node[bag] {$\Big(0,\frac{1}{x^{11}}\Big)$}
child{
        node[bag] {$\Big(\textcolor{black}{11},\textcolor{black}{\frac{dx}{x^{7}(x-1)^5}}\Big)$}
    child {
                node[end, label=right:
                    {$5[t^5]$}] {}
                edge from parent
            }
    child {
        node[bag] {$\Big(\textcolor{black}{17},\textcolor{black}{\frac{dx}{x^4(x-1)^3}}\Big)$}        
            child {
                node[end, label=right:
                    {$3 [t^{10}]$}] {}
                edge from parent
            }
            child {
                node[end, label=right:
                    {$4 [0]$}]{}
                edge from parent
            }
            edge from parent 
            node[above] {$e_1$}
            node[below]  {$1$}
    }
    edge from parent
    node[above] {$e_0$}
    node[below] {$1$}
    };
\end{tikzpicture}
\]

\noindent At each vertex $v$ of the tree, the first component of the pair is the depth conductor at the boundary of the disc associated to $v$. When the depth is positive, the second component is the differential Swan conductor. When the depth is $0$, the second component represents the degeneration of the restriction. The rational number below each edge $e$ is the thickness of the corresponding annulus divided by $p$. The integer on the right at each leaf denotes the conductor of the associated branch point, and inside $[\cdot]$ is the branch point. We often disregard this information as you can see in Example \ref{examplehurwitzobstruction}. One can easily read off from the leaves that $\psi$ has type $[12] \xrightarrow{}[5,4,3]^{\top}$. 
\end{example}

By Proposition \ref{propactiontotree}, the existence of a Hurwitz tree of type $\{h_1, \ldots, h_r\}$ is necessary for the existence of a deformation of corresponding type. It gives us an obstruction for the deformation of $\mathbb{Z}/p$-covers, which we call the \textit{Hurwitz tree obstruction}.

\begin{example}
\label{examplehurwitzobstruction}
Suppose again that $p=5$. If there exists a deformation of type $[5] \xrightarrow{} [3,2]^{\top}$, then the associated Hurwitz tree must have the form below.

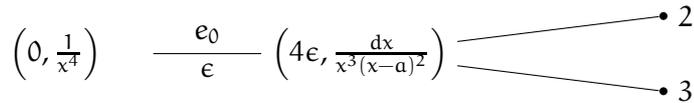
\tikzstyle{level 1}=[level distance=4cm, sibling distance=2cm]
\tikzstyle{level 2}=[level distance=4cm, sibling distance=1cm]

\tikzstyle{bag} = [text width=6em, text centered]
\tikzstyle{end} = [circle, minimum width=3pt,fill, inner sep=0pt]

\[ 
\begin{tikzpicture}[grow=right, sloped]
\node[bag] {$\Big(0,\frac{1}{x^{4}}\Big)$}
child{
        node[bag] {$\Big(\textcolor{black}{4\epsilon},\textcolor{black}{\frac{dx}{x^{3}(x-a)^2}}\Big)$}
    child {
                node[end, label=right:
                    {$3$}] {}
                edge from parent
            }
    child {
                node[end, label=right:
                    {$2$}] {}
                edge from parent
            }
    edge from parent
    node[above] {$e_0$}
    node[below] {$\epsilon$}
    };
\end{tikzpicture}
\]

\noindent Thus, there exists an exact differential form $\omega=\frac{dx}{x^3(x-a)^2}$, where $a \neq 0$. A straightforward calculation shows that the residue of $\omega$ at $0$ is $3/a^4$, which is a contradiction. Therefore, the Hurwitz tree obstruction does not vanish. Hence, there is no deformation of type $[5] \xrightarrow{} [3,2]^{\top}$.

\end{example}

\begin{proposition}
\label{propgooddefthenhurwitz}
Suppose a $\mathbb{Z}/p$-cover $\psi$ of a formal disc gives rise to a Hurwitz tree of type $\{h_1, \ldots, h_r\}$ and with depth zero from the construction at the beginning of \S \ref{seccovertotree}. Then $\psi$ is a deformation of type $[h_1, \ldots, h_r]^{\top}$.
\end{proposition}

\begin{proof}
It follows from the depth zero assumption that $\psi$ has {\'e}tale reduction. Moreover, as the reduction has conductor $h=\sum_{i=1}^r h_i$ by Lemma \ref{lemmarootconductor}, it is smooth by the different criterion (Proposition \ref{propdifferentcriterion}) or Corollary \ref{corgood}.
\end{proof}

\begin{example}
The tree below arises from the deformation in Example \ref{exnonflat} using the data from Remark \ref{remarksummarizeexnonflat}.

\tikzstyle{level 1}=[level distance=4cm, sibling distance=2cm]
\tikzstyle{level 2}=[level distance=4cm, sibling distance=1cm]

\tikzstyle{bag} = [text width=6em, text centered]
\tikzstyle{end} = [circle, minimum width=3pt,fill, inner sep=0pt]
\[ 
\begin{tikzpicture}[grow=right, sloped]
\node[bag] {$\Big(0,\frac{1}{x}\Big)$}
child{
        node[bag] {$\Big(2,{\frac{dx}{x^{2}(x-1)^2}}\Big)$}
    child {
                node[end, label=right:
                    {$2[t^2]$}] {}
                edge from parent
            }
    child {
                node[end, label=right:
                    {$2[0]$}] {}
                edge from parent
            }
    edge from parent
    node[above] {$e_0$}
    node[below] {$1$}
    };
\end{tikzpicture}
\]    
\noindent Note that $\Big(0,\frac{1}{x}\Big)$ (resp. $\Big(2, \frac{dx}{x^2(x-1)^2}\Big)$) is the degeneration type of $\psi$ (resp. $\psi\lvert_{\mathcal{D}[2]}$ ). It violates \ref{c4Hurwitz} and \ref{c5Hurwitz}, as $2=d \neq \ord_{z_{e_0}} \omega_{t(e_0)}+1=4$ and $2=\delta_{t(e_0)} \neq \delta_{s(e_0)}+\epsilon_{e_0} d_{e_0}=3$. Hence, it is not a flat deformation by Proposition \ref{propgooddefthenhurwitz}.
\end{example}

\section{Construction of \texorpdfstring{$\mathbb{Z}/p$}{Zp}-covers from a Hurwitz tree}
\label{sechurwitztocover}

In the previous section, we have associated a Hurwitz tree to a $\mathbb{Z}/p$-cover $\psi: \spec R[[Z]] \xrightarrow{} \spec R[[X]]$. The main result of this section is that this construction can be reversed.

\begin{theorem}
\label{theoremtreeaction}
Every Hurwitz tree $\mathcal{T}=(T, \omega_v, \delta_v, \epsilon_e, h_b, h-1)$ is associated to a $\mathbb{Z}/p$-cover $\psi$ of $\spec R'[[X]]$ for some finite extension $R'$ of $R$.
\end{theorem}

\noindent The proof goes as follows. Using the geometry of the tree $\mathcal{T}$, we partition the open disc $\spec R[[X]]$ into sub-discs, sub-punctured-disc, and annuli as in \S \ref{secsemistablemodel}. In sections \ref{vertex}, \ref{edge}, \ref{leaf}, we construct explicitly a $\mathbb{Z}/p$-cover for each of these pieces that matches its associated datum on the tree. We then glue them together along their boundaries using the technique from \S \ref{glue}. Finally, the details of the proof will be given in \S \ref{proofmain}. 

\subsection{Realization of a vertex}
\label{vertex}

Consider a vertex $v \neq v_0$ on a tree with datum $(\delta_v \in \mathbb{Q}_{>0}, \omega_v=df)$ where $f \in \Frac k[x]$. Suppose $\omega_v$ has $r$ poles $p_1, p_2, \ldots, p_r \in \mathbb{A}^1_k$, and set $d_i:=-\ord_{p_i}(\omega)-1$, $d:=-\ord_{\infty}(\omega)-1$. Then it follows from \ref{c2Hurwitz} that $d+1=\sum_{i=1}^r (d_i+1)$. The points $p_i$ (resp. the point $\infty$) correspond to the singular points $x_{e_i}$ such that $s(e_i)=v$ (resp. to the unique singular point $x_e$ with $t(e)=v$). As in the previous section, $v$ corresponds to a complement of $r$ open discs inside one closed disc
\[ W_K \cong \{X \mid \lvert X \rvert \le 1, \lvert X-P_{i,K} \rvert \ge 1, \} \]
\noindent where $P_{i,K}$ is a lift of $p_i$ to $K$. $W_K$ is the $K$-analytic space of $W=\spec R\{X, (X-P_{i,K})^{-1}\}_{i=1, \ldots, r}$.

Define a $\mathbb{Z}/p$-cover $U \xrightarrow{\psi} W$ by
\begin{equation}
\label{eqnvertex}
    Y^p-Y=\frac{F}{t^{p\delta}},
\end{equation}
\noindent where $F$ is an element of $R\{X, (X-P_{i,K})^{-1}\}_{i=1, \ldots, r}$ that lifts $f$ (hence only have the $P_{i,K}$'s as poles). Let $\spec S_i:=R[[X-P_{i,K}]]\{(X-P_{i,K})^{-1} \}$ (resp. $\spec S_{\infty}:=R[[X^{-1}]]\{X \}$) be the boundary of the missing open disc containing the point $P_{i,K}$ (resp. $\infty$). The first declaration of the below proposition follows easily from Proposition \ref{proprefinedswanorderp} and Remark \ref{remarkboundaryswan}.

\begin{proposition}
\label{propvertex}
The cover $\psi$ restricts to a $\mathbb{Z}/p$-cover on $\spec S_i$ (resp. on $\spec S_{\infty}$) which has boundary conductor $d_i$ (resp. boundary conductor $d$) and depth $\delta$. In fact, the covering space $U$ is itself a punctured disc.
\end{proposition}

\begin{proof}
Set $\mathcal{A}:=R\{X, (X-P_{i,K})^{-1}\}_{i=1, \ldots, r}$ and $\mathcal{B}:=\mathcal{A}[Z]/(Z^p-t^{(p-1)\delta}Z-F)$. The equation defining $\mathcal{B}$ is exactly (\ref{eqnvertex}) with $Z=Y/t^{\delta}$. Hence, $U$ is the spectrum of $\mathcal{B}$. One may think of its reduction, $\overline{\psi}: \spec \overline{\mathcal{B}} \xrightarrow[]{} \spec \overline{\mathcal{A}}$, as a degree-$p$-cover of $\mathbb{P}^1_k \setminus \{ \infty, p_i\}_{i=1,\ldots, r}$ that is defined by the equation $\overline{Z}^p-f$. In addition, as $df=\omega_v$ has neither a zero nor a pole on $\mathbb{P}^1_k \setminus \{ \infty, p_i\}_{i=1,\ldots, r}$ (by \ref{c2Hurwitz}), the scheme $\spec \overline{\mathcal{B}}$ is a smooth superelliptic curve. Thus, we may write $\overline{\mathcal{B}}=k[v,(v-v_i)^{-1}]_{i=1, \ldots, r}$. Therefore, there exists a $V$ (resp. $V_i$) in $\mathcal{B}$ that lifts $v$ (resp. $v_i$) and so that $\mathcal{B}=R\{V, (V-V_i)^{-1} \}_{i=1, \ldots, r}$.
\end{proof}

\subsection{Realization of an edge}
\label{edge}
Suppose $e$ is an edge of thickness $\epsilon$ with final degeneration type $(\delta_1,d_1)$ and initial degeneration type $(\delta_2,-d_2)$, where $d_1 = d_2=:d$ (which is positive as discussed in \S \ref{secgoodannulus}) and $\delta_1-d \epsilon = \delta_2$. The edge corresponds to an annulus $\mathcal{X}$ of thickness $p\epsilon$, which can be identified with $\spec R[[X,U]]/(XU-t^{p \epsilon})$. Consider a $G \cong \mathbb{Z}/p$-cover of $\psi: \mathcal{Y} \xrightarrow{} \mathcal{X}$ defined by
\begin{equation}
\label{eqnrealizeedgeinner}
    Y^p-Y=\frac{X^{d}}{t^{p\delta_1}}.
\end{equation}

\noindent Call $S_1:=\spec R[[X]]\{X^{-1}\}$ the inner boundary of the annulus, and $S_2:=\spec R[[U^{-1}]]\{U\}$ its outer boundary. Then it is immediate that the $\mathbb{Z}/p$-extension of $S_1$ has boundary conductor $d$ and depth $\delta_1$. Replace $X$ by $t^{p \epsilon}/U$ in (\ref{eqnrealizeedgeinner}), we get
\begin{equation}
\label{eqnrealizeedgeouter}
    Y^p-Y=\frac{1}{t^{p(\delta_1-d\epsilon)}U^{d}}.
\end{equation}

\begin{proposition}
\label{propedge}
The covering space $\mathcal{Y}$ is itself an annulus and the followings hold.
\begin{enumerate}
    \item The induced $G$-cover of $S_1$ has conductor $d$ and depth $\delta_1$.
    \item The induced $G$-cover of $S_2$ has conductor $-d$ and depth $\delta_1-d \epsilon$.
\end{enumerate}
\end{proposition}

\begin{proof}
Set $Y=:Y_1/t^{\delta_1}$ and $Y=:Y_2/t^{(\delta_1-d\epsilon)}$. One may rewrite (\ref{eqnrealizeedgeinner}) (resp. (\ref{eqnrealizeedgeouter})) as $Y_1^p-t^{(p-1)\delta_1}Y_1=X^{d}$ (resp. $Y_2^p-t^{(p-1)(\delta_1-d\epsilon)}Y_2=1/U^{d}$). Hence, the covering space of the inner disc $\spec R[[X^{-1}]]$ (resp. the outer disc $\spec R[[U]]$) is also a disc with parameter  $Z_1:=Y_1^{1/d}$ (resp. $Z_2:=Y_2^{-1/d}$). A straightforward computation then confirms that $\mathcal{Y}$ is the spectrum of the $R$-algebra $R[[Z_1,Z_2]]/\allowbreak (Z_1Z_2-t^{\epsilon})$. The rest then immediately follows from Proposition \ref{propdiscbound}.
\end{proof}

\subsection{Realization of a leaf}
\label{leaf}

Consider a leaf $b_i \in B$ on a tree with depth $\delta>0$ and conductor $d_i \not\equiv 0 \pmod{p}$. It corresponds to a singular point $x_i$ and can be associated with a closed disc
\[ C_{i,K} =\{ X \mid \lvert X- x_i \rvert \le 1, \}  \]
\noindent which are $K$-points of $\spec R\{X-x_i\}$. Let $S:=\spec R[[(X-x_i)^{-1} ]] \{X-x_i\}$ be the boundary of $C_{i,K}$. Consider the $\mathbb{Z}/p$-cover $\psi_i$ of $C_{i,K}$ given by the following equation
\[ Y^p-Y=\frac{1}{t^{p \delta} (X-x_i)^{d_i}}. \]
\noindent Then, by applying Proposition \ref{proprefinedswanorderp}, we obtain the following result.

\begin{proposition}
The cover $\psi_i$ of $C_{i,K}$ has depth $\delta$ and boundary conductor $d_i$.
\end{proposition}

\subsection{Gluing the boundaries together}
\label{glue}

In this section, we glue covers of smaller pieces together to form one for a larger piece.

\subsubsection{Filling in a punctured disc}
\label{secfillpunctured}

Suppose $\delta \in \mathbb{Q}_{>0}$, and $d_1, d_2, \ldots, d_r$ are not-divisible-by-$p$ integers. Suppose $\psi_1, \psi_2, \ldots, \psi_r$ are $G \cong \mathbb{Z}/p$-covers of the open discs $\spec R[[Y_1]], \allowbreak \ldots,\allowbreak \spec R[[Y_r]]$, respectively (whose covering spaces are not necessary discs). Suppose $\psi_i$ has depth $\delta$ and boundary conductor $d_i$. It then follows from Proposition \ref{propdiscbound} that we may assume, after a change of variable, that $\psi_i$ induces on $M_i:=\spec R[[Y_i^{-1}]]\{Y_i\}$ a $\mathbb{Z}/p$-cover
$$ Y^p-Y=\frac{1}{t^{p  \delta} Y_i^{d_i}}. $$

\noindent Suppose $\psi$ is a $\mathbb{Z}/p$-cover of a closed punctured disc $\spec R\{X,(X-a_i)^{-1}\}_{i=1,\ldots,r}$ as in \S \ref{vertex}. We may assume that the restriction of $\psi$ to $S_i= \spec R[[X-a_i]]\{(X-a_i)^{-1}\}$, after a change of variable, is given by 
\[ Y^p-Y=\frac{1}{t^{p \delta}(X-a_i)^{d_i}}\]
(possible by Proposition \ref{propdiscbound}). We define the $R$-module isomorphism $\psi_i: R[[(X-a_i)^{-1}]]\{X-a_i\} \xrightarrow{} R[[Y_i^{-1}]]\{ Y_i\}$ by mapping $(X-a_i)$ to $Y_i$. It is clear from the construction that $\psi_i$ and $\psi$ coincides on the glued boundary. 

We would like to construct a $\mathbb{Z}/p$-cover of $\spec R\{X\}$ whose restriction to the closed punctured disc $W$ coincides with $\psi$, and whose restriction to each open disc $\spec R[[X-a_i]]$ is identical to $\psi_i$ (when $\spec R[[Y_i]]$ is identified with $\spec R[[X-a_i]]$). To do that, we fill in $W$ by identifying the boundary of $\spec R[[Y_i]]$ with $S_i$ just like above.

The below lemma shows explicitly how to patch together the disks on the \emph{bottom} and using the compatibilities on the top. 

\begin{lemma}[{\cite[Lemma 3.7]{2000math.....11098H}}]
\label{lemmaglueclosed}
Suppose the elements $a_1, \ldots, a_r$ of $R$ are pairwise distinct modulo $t$. We denote, for each $1 \le i \le r$, $\alpha_i$ (resp. $\beta_i$) the canonical injection of the $R$-algebras  $R\{X, (X-a_j)^{-1}\}_{1 \le j \le r}$ (resp. $R[[Y_i]]$) in $R[[X-a_i]]\{(X-a_i)^{-1}\}$ (resp. $R[[Y_i]]\{Y_i^{-1}\}$). Let $\psi_i: R[[Y_i]]\{Y_i^{-1}\} \xrightarrow{} R[[X-a_i]]\{(X-a_i)^{-1}\}$ be an isomorphism of $R$-algebras.
If $\theta$ is the $R$-module homomorphism 
$$R\{X, (X-a_j)^{-1} \}_{1 \le j \le r} \times \prod_{1 \le i \le r} R[[Y_i]] \xrightarrow{\theta} \prod_{1 \le i \le r} R[[X-a_i]]\{(X-a_i)^{-1}\}$$
$$\theta(f_0, f_1, \ldots, f_r)=(\alpha_1(f_0)-\psi_1 \circ \beta_1(f_1), \ldots, \alpha_r(f_0)-\psi_r \circ \beta_r(f_r)),$$
\noindent for $f_0 \in R\{X, (X-a_j)^{-1} \}_{1 \le j \le r}$ and $f_i \in R[[Y_i]]$ for $1 \le i \le r$, then $\theta$ is surjective and its kernel $N$ is an $R$-algebra $R\{Y_0\}$ as desired.
\end{lemma}

Recall that for each $i$, $\psi_i\lvert_{M_i}$ and $\psi\lvert_{S_i}$ coincides via the identification $X-a_i \mapsto Y_i$. We thus can also glue together the covers $\psi$ and $\psi_i$'s along the lifts of their boundaries $M_i$'s and $S_i$'s on the top (which are boundaries themselves) in a $G$-equivariant way to obtain a well-defined $G$-action on the whole lift. The result is a $\mathbb{Z}/p$-cover of $R\{Y_0\}$ that restricts to $\psi$ on $\spec R\{X, (X-a_i)^{-1}\}_{i=1, 2 \cdot, r}$ and $\psi_i$ on $\spec R[[Y_i]]$ as desired. Note that, the preimage of $R\{Y_0\}$ under $\psi$ is not necessary a disc. To prove that the final cover has good reduction, we only need to show that it gives rise to a Hurwitz tree with depth zero (Proposition \ref{propgooddefthenhurwitz}).  
 
\begin{remark}
\label{remarkHenrioglueingvertex}
In \cite{2000math.....11098H}, Henrio pastes the discs and punctured discs on the \emph{top} together $G$-equivariantly using his Lemma 3.7. We can do the exactly same thing as the covering space of a punctured disc realized in \S \ref{vertex} is also a punctured disc (by Proposition \ref{propvertex}), and one may assume that the preimages of the $\psi_i$'s are likewise discs using induction. 
\end{remark}

\subsubsection{Glueing a closed disc with an annulus}

\label{secglueclosedannulus}

Suppose we are given $\mathbb{Z}/p$-cover $\psi$ of an annulus $\mathcal{X}= \spec R[[X,U]]/ \allowbreak (XU-t^{p \cdot \epsilon})$ as in \S \ref{edge} and a $\mathbb{Z}/p$-cover $\psi'$ of $\spec R\{Y\}$ whose restriction to $\spec R[[Y^{-1}]]\{Y\}$ has depth $\delta_2$ and conductor $m_2$. As the previous subsection, we can use the following lemma to glue the boundary of the closed disc to the inside boundary of the annulus at the bottom to form an open disc. 

\begin{lemma}[{\cite[Lemma 3.8]{2000math.....11098H}}]
\label{lemmaglueopen}
Suppose $e$ is a strictly positive integer, $\beta$ the canonical injection of $R[[X,U]]/(XU-t^{p \cdot e})$ in $R[[U]]\{U^{-1}\}$, $\alpha$ is the canonical injection of $R\{Y\}$ to $R[[Y^{-1}]]\{Y\}$ and $\psi$ is an isomorphism of $R$-algebras $R[[X]]\{X^{-1}\}$ and $R[[Y^{-1}]]\{Y\}$ (by mapping $X$ to $Y^{-1}$). If $\theta$ is a homomorphism of $R$-modules
\[ R\{Y\} \times  \frac{R[[X,U]]}{XU-t^{p \cdot e}} \xrightarrow{} R[[Y^{-1}]]\{Y\}\]
\noindent defined by $\theta(f,g):=\alpha(f)-\psi \circ \beta(g)$, then $\theta$ is surjective, and its kernel $N$ is an algebra $R[[Y_o]]$.

\end{lemma}

As before, by patching together the cover $\psi$ and $\psi'$ using the above gluing on the bottom, we also attain a $\mathbb{Z}/p$-cover of an open unit disc corresponding to $\spec R[[U]]$ that restricts to $\psi$ on $\spec R[[X,U]]/(XU-t^{p \cdot \epsilon})$ and to $\psi'$ on $\spec R\{Y\}$.

\begin{remark}
\label{remarkHenrioglueingedge}
The method of glueing covering spaces on the top (discussed in Remark \ref{remarkHenrioglueingvertex}) can also apply here as we show the preimage of $\mathcal{X}$ is also an annulus (Proposition \ref{propedge}). As in the previous remark, we need the preimage of $\psi'$ to be a closed disc.
\end{remark}


\subsection{Proof of Theorem \ref{theoremtreeaction}}
\label{proofmain}

Let $\mathcal{T}=(C,\omega_v, \delta_v, \epsilon_e,h_b, h-1)$ be a Hurwitz tree with conductor $h-1$ and depth $\delta$. We call a $G=\mathbb{Z}/p$-cover of the open disc a \emph{realization} of $\mathcal{T}$ if $\mathcal{T}$ is associated to this cover by the construction in \S \ref{seccovertotree}. Theorem \ref{theoremtreeaction} claims that we can realize $\mathcal{T}$. We will prove this claim by induction on the height of the tree $\mathcal{T}$. Suppose first that the height of $\mathcal{T}$ is one. Hence, the associated Hurwitz tree has the form like in Figure \ref{treeequidistant}. Let $e_0$ be the trunk of $\mathcal{T}$, we set $v_1=t(e_0)$. Then $C_0:=C_{v_1}$ is the unique component of $C$ which contains the distinguished point $x_0 \in C$. Set $\delta_1:=\delta_{v_1}$, and suppose $b_1, \ldots, b_r$ are the leaves with source $v$, and the $x_i:=x_{e_i}$ are the corresponding singular points associated to the leaves. Let $h_i:=h_{b_i}$. By \S \ref{leaf}, there exists an open unit disc $C_{i,K}$ over some finite extension $K$ of $k((t))$, together with a $G$-cover $\psi_{v_1,i}$, whose depth is $\delta_1$ and whose boundary conductor is $h_i$. Let $\spec S_i$ be the boundary of $C_{i,K}$. By assumption, the differential form $\omega:=\omega_{v_1}$ on $C_0$ has a zero of order $d_{t(e_0)}-1$ at $x_0$ and poles of order $d_i+1$ at $x_i$. Let $W_{v_1,K}$ be the punctured disc with the $G$-cover constructed in Section \ref{vertex}, starting from the datum $(C_0,\omega,\delta_1)$. By Proposition \ref{propdiscbound} and Section \ref{edge}, we can identify the boundary of the missing open disc corresponding to the point $x_i$ with $\spec S_i$ in a way which makes $\psi_{v_1,i}$ and $\psi_{v_1}$ coincides on the boundary. We can now use Lemma \ref{lemmaglueclosed} to patch together the punctured disc $W_{v_1,K}$ and the discs $C_{i,K}$, in a $G$-equivariant way. The result is a closed disc $C_{v_1,K}\cong \{ C_{v_1} \mid \lvert C_{v_1} \rvert \le 1\}$, together with a $G$-cover $\psi_{v_1}$. By Proposition \ref{propvertex} and the construction, the restriction of the cover to the boundary of $C_{v_1,K}$ has conductor $d_{t(e_0)}$ and depth $\delta_1$. 
\begin{figure}[ht]
\tikzstyle{level 1}=[level distance=6cm, sibling distance=2cm]
\tikzstyle{level 2}=[level distance=4cm, sibling distance=1cm]
\tikzstyle{bag} = [text width=9.34em, text centered]
\tikzstyle{end} = [circle, minimum width=3pt,fill, inner sep=0pt]
$$
\begin{tikzpicture}[grow=right, sloped]
\node[bag] [text width=4em] {$\Big(0,\frac{1}{x^{l-1}}\Big)$}
child{
        node[bag][text width=9.34em] {$\Big(\textcolor{black}{\epsilon(l-1)},\textcolor{black}{\frac{dx}{\prod_{i=1}^r(x-P_i)^{l_i}}}\Big)$}
    child {
                node(a)[end, label=right:
                    {$l_r$}] {}
                edge from parent
            }
    child {
                node(b)[end, label=right:
                    {$l_2$}] {}
                edge from parent
            }
    child {
                node[end, label=right:
                    {$l_1$}] {}
                edge from parent
            }
    edge from parent
    node[above] {$e_0$}
    node[below] {$\epsilon$}
    };
\path (a) -- node[auto=false]{\ldots} (b);
\end{tikzpicture}
$$
\caption{A tree of height one}
\label{treeequidistant}
\end{figure}
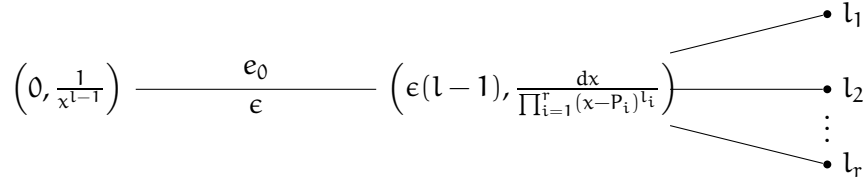

Let $A_K$ be the open annulus with the $G$-cover constructed in Section \ref{edge}, starting from the target datum $(\delta_1, d_{t(e_0)})$ and the initial datum $(\delta, -d_{s(e_0)}=-d)$. By \S \ref{secglueclosedannulus}, we can identify the boundary of $X_{v_1,K}$ with the "inner boundary" (associated with the target datum) of $A_K$, together with the $G$-cover, along these boundaries. The result is a $G$-cover $\psi$ of an open disc with boundary conductor $d$, depth $\delta$. By construction $\psi$ is a realization of the Hurwitz tree $\mathcal{T}$. That completes the base case of the induction.

Therefore, we may assume that we can realize all Hurwitz trees with lower height than $C$. Again, let $e_0$ be the trunk of $\mathcal{T}$, $v_1=t(e_0)$, $\delta_1:=\delta_{v_1}$, and suppose $e_1, \ldots, e_r$ are the edges with source $v$, and $x_i:=x_{e_i}$ the corresponding singular points. Let $d_i:=d_{s(e_i)}$. Let $\mathcal{T}_i \subsetneq \mathcal{T}$ be the subtree of $\mathcal{T}$ which contains the point $x_i$ but not the component $C_0$. It is clear that $\mathcal{T}_i$ inherits from $\mathcal{T}$ the structure of a Hurwitz tree, with conductor $h_i$, depth $\delta_1$, and height strictly less than $\mathcal{T}$. By our induction hypothesis, there exists an open unit disc $C_{i,K}$ over some finite inseparable extension $K$ of $k((t))$, together with a $G$-cover $\psi_{v_1,i}$, whose associated Hurwitz tree is $\mathcal{T}_i$. We than can apply the same process as in the base case to construct a $G$-cover $\psi$ that realizes the tree $\mathcal{T}$. This completes the proof of the theorem. 
\qed

\begin{remark}
In \cite{MR2254623}, Bouw and Wewers generalize the notion of Hurwitz tree from \cite{2000math.....11098H} to the case where $G=\mathbb{Z}/p \rtimes_{\chi} \mathbb{Z}/m$ ($m$ is prime to $p$) by adding an extra piece of information, which they call the \textit{tame inertia character}, to Henrio's tree. It comes from the group homomorphism $\chi: \mathbb{Z}/m \xrightarrow{} \aut(\mathbb{Z}/p)$ that defines the semi-direct product. Furthermore, they prove an analog of \cite[Corollary 1.8]{2000math.....11098H} and Proposition \ref{propdiscbound}, which says a $G$-action on a boundary of a disc is determined by its depth, its boundary conductor, and its tame inertia character \cite[Proposition 2.3]{MR2254623}. Using this fact, they utilize Henrio's technique to prove that every local $D_p$-cover in characteristic $p$ lifts to characteristic zero \cite[Theorem 4.4]{MR2254623}. We would expect an analogous result for the equal characteristic case using a parallel theory.
\end{remark}

\subsection{Hurwitz trees and deformations}

\label{secgoodtreedeformation}
Combining Theorem \ref{theoremtreeaction} and Proposition \ref{propgooddefthenhurwitz}, we acquire the following result.

\begin{corollary}
\label{corgoodtreedeformation}
Let $\phi: Z_k \xrightarrow{} X_k$ be a local $G$-cover with conductor $h-1$. Then there exists a deformation of $\phi$ over $k[[t]]$ of type $[h] \xrightarrow{} [h_1, \ldots, h_r]^{\top}$ if and only if there exists a Hurwitz tree of type $\{h_1, \ldots, h_r\}$ and with depth zero.
\end{corollary}

\begin{remark}
The only difference between Corollary \ref{corgoodtreedeformation} and Theorem \ref{theoremmainlocal} is that we still need the Hurwitz tree in the statement to have depth zero. That requirement will be abolished by Proposition \ref{propchaindiffforms}.
\end{remark}









\begin{definition}
Suppose $\overrightarrow{A} \prec \overrightarrow{B}$. Then one can easily see that there exist a minimal submulti-set $\overrightarrow{A}'$ of $\overrightarrow{A}$ and $\overrightarrow{B}'$ of $\overrightarrow{B}$, such that $\overrightarrow{A}' \prec \overrightarrow{B}'$ and $\overrightarrow{A}\setminus \overrightarrow{A}'$ is the same as $\overrightarrow{B} \setminus \overrightarrow{B}'$. We call $\overrightarrow{A}' \prec \overrightarrow{B}'$ the \textit{difference} of $\overrightarrow{A} \prec \overrightarrow{B}$. 
\end{definition}

\begin{example}
The difference of $\{3,4,5,6\} \prec \{3,2,2,3,2,6 \}$ is $\{4,5\} \prec \{2,2,3,2\}$.
\end{example}

The next result shows that the existence of a Hurwitz tree equates to the existence of a chain of exact differential forms of certain type. This phenomenon does not generalize to any cyclic group $G$, though.

\begin{proposition}
\label{propchaindiffforms}
There exists a Hurwitz tree of type $\{h_1, \ldots, h_r\}$ if and only if there exists a chain of partitions of integers $\underline{e}_0:=\{e\} \prec \underline{e}_1 \prec  \ldots \prec \underline{e}_m:=\{h_1, \ldots, h_r\}$, where the difference between $\underline{e}_{i-1}$ and $\underline{e}_{i}$ is $\{l_i\} \prec \{l_{i,1}, \ldots, l_{i,r_i} \}$, together with $m$ exact differential forms
\begin{equation}
    \label{eqndeformationdiff}
    \frac{dx}{\prod_{j=1}^{r_i}(x-P_{i,j})^{l_{i,j}}}.
\end{equation} 
for $1 \leq i \leq m$.
\end{proposition}

\begin{proof}
The "$\Rightarrow$" direction is easily deduced from Corollary \ref{corgoodtreedeformation} where each differential form corresponds to the differential conductor at a vertex $v\neq v_0$ of the associated tree. For instance, one can obtain from the tree in Example \ref{excalculatehurwitzmain} a chain $\{12\} \prec \{ 5, 7 \} \prec \{ 5, 4, 3\}$ and two exact differential forms $\frac{dx}{x^5(x-1)^7}, \frac{dx}{x^4(x-1)^3}$.

"$\Leftarrow$" We prove by induction on $m$. Suppose first that $m=1$. Then, given an exact differential form of type (\ref{eqndeformationdiff}), the tree in Figure \ref{treeequidistant} works. Suppose the statement is true for $m=n$. By induction, there exists a Hurwitz tree $\mathcal{T}'$ with $n$ non-root-vertices whose differential conductors are exact of type (\ref{eqndeformationdiff}). Suppose that the difference between $\underline{e}_n$ and $\underline{e}_{n+1}$ is $\{l\} \prec \{l_1, \ldots, l_r\}$. Suppose, moreover, that there exists an exact differential form $\omega=dx/
\big(\prod_{i=1}^r (x-P_i)^{l_i}\big)$, and the depth at the leaf corresponding to $l$ is $\delta$. Then one can "lengthen" $\mathcal{T}'$ by "thickening" the leaf corresponding to $l$ to an edge of thickness $1$ and adding $r$ leaves of conductors $l_1, \ldots, l_r$ to the end of this edge. Finally, we assign the exact differential form $\omega$ and the depth $\delta+l-1$ to the end of the edge. One can easily check that the extended tree is Hurwitz. 
\end{proof}

\begin{example}

The tree below extends one from Figure \ref{treeequidistant}, 
\tikzstyle{level 1}=[level distance=3cm, sibling distance=2cm]
\tikzstyle{level 2}=[level distance=3cm, sibling distance=1.3cm]
\tikzstyle{level 3}=[level distance=3cm, sibling distance=1.3cm]
\tikzstyle{bag} = [text width=11.5em, text centered]
\tikzstyle{end} = [circle, minimum width=3pt,fill, inner sep=0pt]
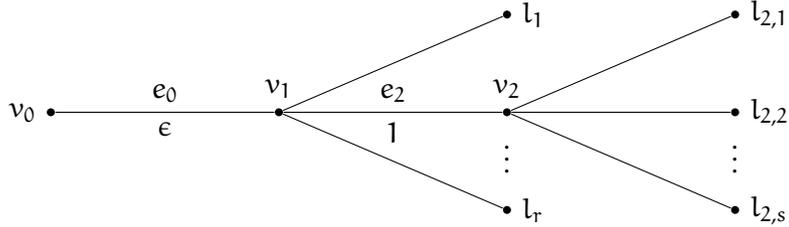
\[
\begin{tikzpicture}[grow=right, sloped]
\node[end, label=left:{$v_0$}]{}
child{
        node[end, label=above:{$v_1$}]{}
    child {
                node(b)[end, label=right:
                    {$l_r$}] {}
                edge from parent
            }
    child {
        node(a)[end, label=above:{$v_2$}]{}        
            child {
                node(d)[end, label=right:
                    {$l_{2,s}$}] {}
                edge from parent
            }
            child {
                node(c)[end, label=right:
                    {$l_{2,2}$}] {}
                edge from parent
            }
            child {
                node[end, label=right:
                    {$l_{2,1}$}] {}
                edge from parent
            }
            edge from parent 
            node[above] {$e_2$}
            node[below]  {$1$}
    }
    child {
                node[end, label=right:
                    {$l_1$}] {}
                edge from parent
            }
    edge from parent
    node[above] {$e_0$}
    node[below] {$\epsilon$}
    };
\path (a) -- node[auto=false]{\ldots} (b);
\path (c) -- node[auto=false]{\ldots} (d);
\end{tikzpicture}
\]
\noindent where $l_2=\sum_{j=1}^s l_{2,j}$, and the degeneration types at $v_0, v_1,$ and $v_2$ are $\Big(0, \frac{1}{x^l}\Big), \Big(\epsilon(l-1), \frac{dx}{\prod_{i=1}^r(x-P_i)^{l_i}} \Big)$, and $\Big(\epsilon(l-1)+l_2-1, \frac{dx}{\prod_{j=1}^s(x-P_{2,j})^{l_{2,j}}} \Big)$, respectively.
\end{example}

\begin{remark}
Proposition \ref{propchaindiffforms} implies that one can drop the "depth zero" condition in Corollary \ref{corgoodtreedeformation}. Theorem \ref{theoremmainlocal} then easily follows. 
\end{remark}

We call an Artin-Schreier cover \emph{equidistant} if the distance between any two branch points is the same. A deformation of a one point cover is equidistant if its generic fiber is equidistant, hence its Hurwitz tree has height one. Most of the known deformations have equidistant deformations as building blocks. The follows results then follows immediately from Proposition \ref{propchaindiffforms}.

\begin{corollary}
\label{corequidistant}
Suppose $\{h_1,h_2, \ldots,h_r\}$ is a partition of $\{h\}$. Then there is a equidistant deformation over $k[[t]]$ of type $[h] \xrightarrow{} [h_1, \ldots, h_r]^\top$ if and only if there exists an exact differential of the form
\begin{equation}
\label{equidiffform}
    \frac{dx}{\prod_{i=1}^r(x-q_i)^{h_i}}, 
\end{equation}
 \noindent for some distinct elements $q_i$'s in $k$. 
\end{corollary}

\begin{definition}
\label{deftypeofdeform}
We say the exact differential form in (\ref{equidiffform}) has \textit{type} $\{h_1, \ldots, h_r\}$.
\end{definition}

In the remainder of the paper, we will focus on differential forms of type (\ref{equidiffform}).




\begin{remark}
\label{remarkreducetypediff}
Suppose we are given a rational function of the form $f(x)=g(x)^p h(x)$. Then $f'(x)=g(x)^p h'(x)$. Hence, it suffices to study differential forms of type (\ref{equidiffform}), where $1<h_i<p$.
\end{remark}

\begin{remark}
In \cite{MR1767273}, Bertin and M{\'e}zard show that the infinitesimal deformation functor of a local Galois cover is represented by a \textit{versal deformation ring}. Besides, they give some descriptions for the versal deformation rings of $\mathbb{Z}/p$-covers. Theorem 3.5 of \cite{2000math.....11098H} and Proposition \ref{propchaindiffforms} imply that one can describe a deformation (over either a mixed or equal characteristic ring) of a $\mathbb{Z}/p$-cover by the associated Hurwitz tree, which, in turn, is determined by a sequence of exact differential forms in equal characteristic case, or exact and logarithmic differential forms in mixed characteristic case. We wonder if one can use a space of differential forms to give a full description of the versal deformation ring of an Artin-Schreier cover. 
\end{remark}

\section{Applications}
\label{secapplication}

\subsection{Some general exact differential form results}
\label{secconnection}

Recall from Proposition \ref{propchaindiffforms} that the deformation of an Artin-Schreier cover is determined by exact differential forms of type (\ref{eqndeformationdiff}). Combining with Remark \ref{remarkreducetypediff}, we simplify Question \ref{questionlocalgoodreduction} as follows.

\begin{question}
\label{questiondifftype}

Let $k$ be an algebraically closed field of characteristic $p>0$. Suppose $1< h_i <p$ for $i=1,2, \ldots, n$ are integers ($n \ge 2$). What are the conditions on the $h_i$'s so that the rational function
\begin{equation}
    \label{eqnexactquestion}
    \frac{1}{\prod_{i=1}^n (x-P_i)^{h_i}}
\end{equation}
\noindent is a derivative of some rational function in $k(x)$ for some $P_i$'s in $k$ pairwise distinct?
\end{question}

When the multi-set $\{h_1, \ldots, h_n\}$ is given, we can answer the question using Gr{\"o}bner Basis techniques as follows. Equation (\ref{eqnexactquestion}) could be written as
\begin{equation}
\label{eqngrobnersetup}
    \frac{\prod_{i=1}^n(x-P_i)^{p-h_i}}{\prod_{i=1}^n(x-P_i)^{p}}=\frac{\sum_j a_j x^j}{\prod_{i=1}^n(x-P_i)^{p}},
\end{equation}
where $a_j$ can be thought of as a polynomial in $k[P_1, \ldots, P_n]$. As the denominator is a $p$-power, the fraction is exact if and only if the numerator $\sum_j a_j x^j$ is exact. That equates to there existing a choice of values for the $P_i$'s so that $a_j=0$ for all $j \equiv -1 \pmod{p}$, or $1 \not \in (a_j)_{j \equiv -1 \pmod{p}}$ by Hilbert's nullstellensatz. Moreover, all the $P_i$'s have to be distinct. That translates to $\prod_{j<k} (P_j-P_k)$ not lying in the radical of $(a_j)_{j \equiv -1 \pmod{p}}$, or the following ideal $$\Big(\{a_j\}_{j \equiv -1 \pmod{p}}, 1-s \cdot \prod_{j<k} (P_j-P_k)\Big)$$  
is not the unit ideal of $k[P_1, \ldots, P_n, s]$ by \cite[\S 15, Corollary 35]{MR2286236}. We summarize by the below proposition.

\begin{proposition}
\label{propgrobner}
Suppose we are given a multi-set $\{h_1, \ldots, h_n\} \in \Omega_{h}$. Then there exists a differential form
\[ \frac{dx}{\prod_{i=1}^n(x-P_i)^{h_i}} \]
\noindent if and only if $\big(\{a_j\}_{j \equiv -1 \pmod{p}}, 1-s \prod_{j<k} (P_j-P_k)\big) \in k[P_1, \ldots, P_n, s]$ is not the unit ideal, where the $a_j$'s are defined in (\ref{eqngrobnersetup}).
\end{proposition}

\begin{remark}
\label{remarkgrobner}

We can show whether the ideal in Proposition \ref{propgrobner} is a unit one by checking whether $1$ is its reduced Gr{\"o}bner basis (with respect to any monomial ordering).

\end{remark}

\begin{proposition}
\label{propfirstcombinatorial}
Let $k$ be an algebraically closed field of characteristic $p>0$. Suppose $1< h_i <p$ for $i=1,2, \ldots, n$ are integers ($n \ge 2$). The rational differential form 
\begin{equation}
    \label{eqncanexact}
    \omega=\frac{dx}{\prod_{i=1}^n (x-P_i)^{h_i}}
\end{equation}
\noindent is exact for some distinct $P_i$'s in $k$ only if $\sum_{i=1}^n h_i \ge p+n$. The converse is true when $n=2$.
\end{proposition}

\begin{proof}
This argument is due to Fedor Petrov (see \S \ref{secacknowledge}). Suppose that $\sum_{i=1}^n h_i<n+p$ and the differential form $\omega$ in (\ref{eqncanexact}) is exact. One may assume that the rational antiderivative of $\omega$ is of the form $g/f$, where $f=\prod_{i=1}^n (x-P_i)^{h_i-1}$, and $\deg g<\deg f=\sum_{i=1}^n (h_i-1)<p$. This may be seen from integrating the partial fraction decomposition of $\prod_{i=1}^n (x-P_i)^{-h_i}$ and convert it back to a single fraction. We have $(g/f)'=(g'f-f'g)/f^2$, and if $\deg f=a$, $\deg g=b$, the degree of the numerator equals $a+b-1$, since the leading coefficient does not vanish (here we use that $p$ can not divide $a-b$). Thus $1=\prod_{i=1}^n (x-P_i)^{h_i} (g/f)'$ has degree $(n+a)+(a+b-1)-2a=n+b-1>0$, a contradiction. Therefore, $\sum_{i=1}^n h_i$ has to be at least $n+p$ for (\ref{eqncanexact}) to be exact.

Suppose $n=2$. One may assume that $P_1=0$. Consider the rational function
 $$ \omega=\frac{1}{x^{e_1}(x-Q)^{e_2}}=\frac{x^{p-e_1}(x-Q)^{p-e_2}}{ x^p(x-Q)^p}. $$
\noindent Hence, $\omega$ is a derivative of some rational function if and only if $x^{p-e_1}(x-Q)^{p-e_2}$ is. The later is exact if and only if all the degree $kp-1$ coefficients are equal to $0$ as $k$ varies. Suppose $e_1+e_2 \ge p+2$. Then the degree of the numerator is $2p-(e_1+e_2) \le p-2$. Thus, it is a derivative. 
\end{proof}

\begin{remark}
The converse is not true in general for $n>2$. See \url{https://mathoverflow.net/questions/310575/residues-of-frac1-prod-i-1n-x-p-ie-i} for some counterexamples to differential forms of type $\{2,2, \ldots, 2\}$ by Gjergji Zaimi. Using the Gr{\"o}bner Basis technique (Proposition \ref{propgrobner} and Remark \ref{remarkgrobner}), one can also show that there is no exact differential form of type $\{2,2,2,6\}$ where $p=7$, even though $2+2+2+6>7+4$.
\end{remark}

\subsubsection{Oort-Sekiguchi-Suwa deformations}
\label{secOSS}

We first observe the following phenomenon.

\begin{proposition}
\label{propOSStype}
The differential form
\begin{equation}
\label{eqndiffOSS}
     \frac{dx}{\prod_{i=1}^r(x-q_i)^{h_i}},
\end{equation}
\noindent where $q_i \in k$'s are pairwise distinct, is exact if all but at most one of the $h_i$'s are divisible by $p$. In particular, there exists a flat deformation of type $[h] \xrightarrow{} [h_1, \ldots, h_r]^{\top}$.
\end{proposition}

\begin{proof}
Suppose, without loss of generality, that $h_1$ is the only exponent that may not be divisible by $p$. Then we can easily see that differential form (\ref{eqndiffOSS}) is the derivative of the following rational function
\[ \frac{1}{(1-h_1)\prod_{i=2}^{r}(x-q_i)^{h_i} (x-q_1)^{h_1-1}}. \]
\noindent The rest of the proposition follows from Proposition \ref{propchaindiffforms}.
\end{proof}

\begin{remark}
Suppose $h_i$'s are as in Proposition \ref{propOSStype}. Consider the $\mathbb{Z}/p$-cover of $\mathbb{P}^1_{k[[t_1, \ldots, t_r]]}$ defined by
\begin{equation}
    \label{eqnOSS}
     Y^p-Y= \frac{1}{(x-t_1)^{h_1-1}\prod_{i=2}^r (x-t_i)^{h_i}} 
\end{equation}
\noindent Replacing the $t_i$'s by any distinct elements $P_i$'s of $\mathfrak{m}k[[t]]$ (i.e. $v_t(P_i)>0$), one can show that (\ref{eqnOSS}) defines a flat deformation (over $k[[t]]$) of type $[h] \xrightarrow{} [h_1, h_2, \ldots, h_r]^{\top}$.
\end{remark}

\begin{remark}
Suppose $\Psi$ is a $\mathbb{Z}/p$-cover over $k[[t]]$. In \cite{MR1767273}, Bertin and M\'{e}zard study the \emph{versal deformation ring} $R_{\Psi}$ whose spectrum is the formal deformation space of $\Psi$. They construct explicitly a family of deformations of $\Psi$ over a polynomial ring over $R=\mathcal{W}(k)[\zeta_p]$. This family is parameterized by an irreducible component of the formal deformation space, called \emph{the Oort-Sekiguchi-Suwa component}. Theorem 5.3.3 of the same paper shows that the dimension of this component is equal to the dimension of the versal deformation ring. Following \cite[\S 3.1.1]{DANG2020398}, one can realize the characteristic $p$ fibers of these deformations as special cases of one described in Proposition \ref{propOSStype}.  
\end{remark}

\subsubsection{Non Oort-Sekiguchi-Suwa deformations}

In \cite{DANG2020398}, to prove that $\mathcal{AS}_g$ is connected when $g$ is large, we construct some equidistant deformations that do \textit{not} lie in $p$-fibers of the Oort-Sekiguchi-Suwa component. In this section, we realize these deformations in term of exact differential forms.

\begin{proposition}
\label{propnonOSStype}

There exist exact differential forms of the following types:
\begin{enumerate}
    \item \label{firstdeformation} $\{h_1, h_2\}$ where $\overline{h}_1+\overline{h}_2 \ge p+2$,
    \item \label{seconddeformation}$\{h_1, h_2, h_3\}$ where $h_i \equiv (p+1)/2 \pmod{p}$ and $p\ge3$,
    \item \label{thirddeformation}$\{p-1, h_1, h_2\}$ where $h_1, h_2 \not \equiv 1 \pmod{p}$,
    \item \label{fourthdeformation} $\{\underbrace{n+1, n+1, \ldots,n+1}_{p-n+1}\}$ where $1 \le n \le p-1$
    \item \label{fifthdeformation}$\{ 3,2,2,2 \}$ and $\{3,3,2,2 \}$ where $p=5$
\end{enumerate}
\end{proposition}

\begin{proof}
Item (\ref{firstdeformation}) follows immediately from \ref{propfirstcombinatorial}.

Consider item (\ref{seconddeformation}). Suppose the desired differential is of the form
\begin{equation}
\label{equiform2}
    \frac{dx}{(x-a_1)^{h_1}(x-a_2)^{h_2}(x-a_3)^{h_3}}.
\end{equation}
As discussed in Remark \ref{remarkreducetypediff}, one may assume, without loss of generality, that $h_i=(p+1)/2$. Set $a_1=0$, $a_2=1$, and rewrite the differential form (\ref{equiform2}) as
\[ \frac{x^{(p-1)/2}(x-1)^{(p-1)/2}(x-a_3)^{(p-1)/2} dx}{x^p(x-1)^p(x-a_3)^p}. \]
As the denominator is a $p$-power, the differential form is exact if and only if all the terms of degree congruent to $-1$ modulo $p$ in the numerator are zero. One can easily see the leading term has degree $3(p-1)/2<2p-1$. Hence, it suffices to make the term of degree $p-1$, which is

\begin{equation}
\label{eqntype2deformation}
    \sum_{i+j=\frac{p-1}{2}} {{\frac{p-1}{2}}\choose {i}} {{\frac{p-1}{2}}\choose {j}} a_3^{(p-1)/2-j}=\sum_{i=0}^{(p-1)/2}{{\frac{p-1}{2}}\choose {i}}^2 a_3^{(p-1)/2-i},
\end{equation}
to be equal to zero. We thus want $a_3$ to be a root of (\ref{eqntype2deformation}) (which can be thought of as a polynomial in $k[a_3]$) that is different from $0$ and $1$. As the constant coefficient of the polynomial is nonzero, $0$ cannot be a root. Moreover, since the polynomial $\sum_{i=0}^{(p-1)/2} {{\frac{p-1}{2}}\choose {i}}^2 x^{(p-1)/2-i}$ is separable by \cite[Theorem 4.1]{MR817210} and $(p-1)/2$ is at least $2$ for $p \ge 5$, there are roots (in $k$) of (\ref{eqntype2deformation}) that are different from $1$. Hence, if we pick $a_3$ to be one of these roots, the differential form is exact as desired.

By direct computation, one can show that the differential forms $\frac{dx}{x^{p-1}(x-1)^{h_1}(x-(1-h_2)/(h_1-1))^{h_2}}$, $ \frac{dx}{x^{n+1}(x^{p-n}-1)^{n+1}}$, $ \frac{dx}{x^3(x-1)^2(x^2+x+1)^2}$, and $\frac{dx}{x^3(x-1)^3(x^2+4x+2)^2}$ are solutions to (\ref{thirddeformation}), (\ref{fourthdeformation}), and (\ref{fifthdeformation}), respectively. \end{proof}

\noindent This gives an alternative proof for \cite[Theorem 3.7]{DANG2020398}. We paraphrase the statement of that theorem using the language in this paper.

\begin{theorem}[{c.f. \cite[Theorem 3.7]{DANG2020398}}] There exist equidistant deformations over $k[[t]]$ of the following type.
\begin{enumerate}
    \item $[h] \xrightarrow{} [h_1, h_2]^{\top}$ where $\overline{h}_1$ and $\overline{h}_2$ are non-zero and $\overline{h}_1+\overline{h}_2 \ge p+2$.
    \item $[h] \xrightarrow{} [h_1, h_2, h_3]^{\top}$, where $h_i \equiv (p+1)/2 \pmod{p}$ and $p\ge 3$.
    \item $[h] \xrightarrow{} [p-1, h_2, h_3]^{\top}$ where $\overline{h}_1$ and $\overline{h}_2$ are non-zero.
    \item $[(n+1)(p-n+1)] \xrightarrow{} [\underbrace{n+1, \ldots, n+1}_{p-n+1}]^{\top}$, where $1 \le n \le p-1$.
    \item $[9] \xrightarrow{} [3,2,2,2]^{\top}$ or $[10] \xrightarrow{} [3,3,2,2]^{\top}$ where $p=5$. 
\end{enumerate}
\end{theorem}

\begin{proof}
The theorem follows immediately from Proposition \ref{propchaindiffforms} and Proposition \ref{propnonOSStype}.
\end{proof}

\subsection{Disconnectedness of \texorpdfstring{$\mathcal{AS}_g$}{asg}}
\label{secgeometryASg}

Recall that the moduli space $\mathcal{AS}_g$ can be partitioned by strata that are associated to partitions of $2g/(p-1)+2$. Moreover, Pries and Zhu show that the dimension of the stratum indexed by $\overrightarrow{E}=\{h_1, \ldots, h_r \}$ is given by
\[ d-1-\sum_{j=1}^r \bigg( \bigg\lfloor \frac{h_j-1}{p} \bigg\rfloor \bigg) \]
\noindent \cite[Corollary 3.11]{MR2985514}, recall that $d=\frac{2g}{p-1}+2$. Therefore, the irreducible components of $\mathcal{AS}_g$ are the closure of the strata indexed by partitions of the form $\{h_1, \ldots, h_r\}$, where $h_i \le p$ for all $i$. One key ingredient of our first connectedness result is that we prove, where the sum of conductors $d+2$ is not congruent $1$ modulo $p$, that all the strata of non-zero-codimension lie in the same connected component \cite[Corollary 4.3]{DANG2020398}. In general, $\mathcal{AS}_g$ is connected only if none of the strata of codimension-zero is closed (unless there is only one stratum in $\mathcal{AS}_g$). Furthermore, the closedness of a stratum can be realized by the following result.

\begin{proposition}
\label{propcheckclosure}

A stratum of $\mathcal{AS}_g$ indexed by $\overrightarrow{E}=\{h_1, \ldots, h_r\}$ is not closed if and only if there exists a partition $\{g_1, \ldots, g_s\} \subseteq \{ h_1, \ldots, h_r\}$, and an exact differential form of type $\{g_1, \ldots, g_s\}$.

\end{proposition}

\begin{proof}

"$\Leftarrow$": Suppose, without loss of generality, that $\{g_1, \ldots, g_s\}=\{h_1, \ldots, h_s\}$. The "$\Leftarrow$" direction then follows immediately from Proposition \ref{propchaindiffforms} as the closure of $\Gamma_{\overrightarrow{E}}$ contains the stratum indexed by $\{\sum_{i=1}^s h_i, h_{s+1}, \ldots, h_r\}$.

"$\Rightarrow$": Suppose $\overrightarrow{E}$ is not closed. Then, by Proposition \ref{propdeformclosure}, there exists a partition $\overrightarrow{E}'=\{l_1, \ldots, l_m\} \prec \overrightarrow{E}$ such that $\Gamma_{\overrightarrow{E}'} \subsetneq \overline{\Gamma}_{\overrightarrow{E}}$. It then follows from Proposition \ref{propreduce} that there exist $1 \le i \le m$ so that $\{l_i\} \prec \{h_{i,1}, \ldots, h_{i,m} \} \subseteq \overrightarrow{E}$, and a deformation of type $[l_i] \xrightarrow{} [h_{i,1}, \ldots, h_{i,m}]^{\top}$. Finally, Proposition \ref{propchaindiffforms} shows that there exists an exact differential form of type $\{g_1, \ldots, g_s\} \subseteq \{h_{i,1}, \ldots, h_{i,m}\}$, where $\{\sum_{i=1}^s g_i\} \prec \{g_1, \ldots, g_s\}$ is the difference between $\overrightarrow{E}$ and the immediate one below in the chain of partitions from that proposition. That completes the proof.
\end{proof}

We can finally give the proof of Theorem \ref{improveconnectedness}. 

\subsubsection{Proof of Theorem \ref{improveconnectedness}}
\label{secproofofimproveconnectedness}
It is already known from Theorem \ref{thmconnected} that, when $p=5$, $\mathcal{AS}_g$ is connected when $g>(p-1)(p-2)=12$. The same result shows that, when $p \ge 5$, $\mathcal{AS}_g$ is disconnected when $(p-1)/2<g \le (p-1)^2/2$. Hence, it suffices to show that $\mathcal{AS}_g$ is disconnected if $(p-1)^2/2<g \le (p-1)(p-2)$, or the sum of conductors $d+2$ is between (and including) $p+3$ and $2p-2$. 

Recall that the irreducible components of $\mathcal{AS}_g$ are indexed by the closure of the strata corresponding to $\{h_1, \ldots, h_r\}$ where $h_i \le p$ for all $i$. Suppose $d+2$ is odd (resp. even). Then there exists a stratum indexed by $\overrightarrow{E}_1:=\{3,2, \ldots, 2\}$ (resp. $\overrightarrow{E}_2:=\{2,2, \ldots, 2\}$). It is straightforward to check that, for any $\{k_1, \ldots, k_s\} \subseteq \overrightarrow{E}_1$ (resp. $\overrightarrow{E}_2$), $\sum_{j=1}^s k_j < p+s$. Here we use that fact that the sum of the entries of $\overrightarrow{E}_1$ (resp. $\overrightarrow{E}_2$), i.e., the number $d+2$, is smaller than $2p-2$. Therefore, it follows from Proposition \ref{propcheckclosure} that the closure of the stratum corresponding to $\{3,2, \ldots, 2\}$ (resp. $\{2,2, \ldots,2\}$) only contains itself. Thus, the moduli space $\mathcal{AS}_g$ is disconnected. \qed


\bibliographystyle{alpha}
\bibliography{mybib}


\end{document}